\documentclass[12pt,leqno]{article}

\usepackage{amsmath,amssymb,amsthm}
\usepackage[all]{xy}
\usepackage{lscape}

\makeatletter

\newtheorem{theorem}{Theorem}[section]
\newtheorem{proposition}[theorem]{Proposition}
\newtheorem{lemma}[theorem]{Lemma}
\newtheorem{corollary}[theorem]{Corollary}

\theoremstyle{definition}

\newtheorem{example}[theorem]{Example}
\newtheorem{definition}[theorem]{Definition}

\theoremstyle{remark}

\theoremstyle{remark}
\newtheorem{remark}[theorem]{Remark}

\def\({{\rm (}}
\def\){{\rm )}}

\let\Mathrm\operator@font
\let\Cal\mathcal
\let\Bbb\mathbb

\def\standop#1{\mathop{\Mathrm #1}\nolimits}
\def\difstop#1#2{\expandafter\def\csname #1\endcsname{\standop{#2}}}
\def\defstop#1{\difstop{#1}{#1}}

\defstop{AB}
\defstop{ann}
\defstop{Ass}
\def\alg{_{\Mathrm{alg}}}

\defstop{CMFI}
\defstop{codim}
\defstop{Coh}
\defstop{coht}
\defstop{Coker}
\defstop{Cone}
\defstop{Cl}
\defstop{Cox}
\defstop{cosk}

\defstop{depth}
\defstop{Div}
\defstop{div}

\defstop{EM}
\defstop{embdim}
\defstop{End}
\defstop{ev}
\defstop{Ext}

\defstop{Flat}
\defstop{Func}
\defstop{Fpqc}

\defstop{Gal}
\def\GL{\text{\sl{GL}}}
\defstop{GD}
\defstop{Good}

\difstop{height}{ht}
\defstop{Hom}

\def\id{\mathord{\Mathrm{id}}}

\difstop{Image}{Im}
\defstop{Inv}
\defstop{ind}

\defstop{Ker}

\defstop{Lch}
\defstop{length}
\defstop{Lin}
\defstop{Lqc}
\defstop{lqc}
\defstop{LQ}
\defstop{LM}
\defstop{Lie}

\def\mon{^\mathrm{mon}}
\defstop{Mat}
\defstop{Max}
\defstop{Min}
\defstop{Mod}
\defstop{Mor}
\defstop{MCM}
\defstop{Map}

\defstop{Nerve}
\defstop{NonSFR}
\defstop{NonCMFI}
\defstop{NonNor}
\defstop{Nor}

\defstop{ob}
\def\op{^{\standop{op}}}

\defstop{PA}
\defstop{PM}
\defstop{Proj}
\defstop{Prin}
\defstop{Pic}

\defstop{Ps}

\defstop{Qch}
\defstop{qch}

\defstop{rad}
\defstop{rank}

\defstop{res}
\defstop{Reg}
\defstop{Ref}

\defstop{Sh}
\def\Sch{\underline{\Mathrm Sch}}

\defstop{Spec}
\defstop{supp}
\defstop{Supp}
\defstop{Sym}
\defstop{Sing}
\defstop{SFR}
\defstop{Soc}

\defstop{Ob}
\def\op{^{\mathrm{op}}}
\def\sep{_{\Mathrm{sep}}}

\difstop{tdeg}{trans.deg}
\def\tor{_{\Mathrm{tor}}}
\def\tf{_{\Mathrm{tf}}}
\defstop{Tor}
\difstop{trace}{tr}

\defstop{Zar}

\def\C{\Cal C}

\def\F{\Cal F}
\def\G{\Cal G}

\def\L{\Cal L}
\def\M{\Cal M}
\def\N{\Cal N}
\def\O{\Cal O}
\def\P{\Cal P}

\let\indlim\varinjlim

\def\ru#1{\lceil #1 \rceil}

\def\uHom{\mathop{\text{\underline{$\Mathrm Hom$}}}\nolimits}

\def\uH{\mathop{\text{\underline{$H$}}}\nolimits}

\def\sdarrow#1{\downarrow\hbox to 0pt{\scriptsize$#1$\hss}}
\def\suarrow#1{\uparrow\hbox to 0pt{\scriptsize$#1$\hss}}
\def\ssearrow#1{\searrow\hbox to 0pt{\scriptsize$#1$\hss}}

\def\section{\@startsection{section}{1}{\z@ }%
  {-3.5ex plus -1ex minus -.2ex}{2.3ex plus .2ex}{\bf }}

\long\def\refname{\par\kern -3ex
  \begin{center}\rm R\sc{eferences}\end{center}\par\kern 
  -2ex}

\def\@seccntformat#1{\csname the#1\endcsname.\quad}

\def\@@@sect#1#2#3#4#5#6[#7]#8{%
  \ifnum #2>\c@secnumdepth 
  \def \@svsec {}\else \refstepcounter {#1}%
  \def\@svsec{}
  \fi 
  \@tempskipa #5\relax 
  \ifdim \@tempskipa >\z@ 
  \begingroup #6\relax \@hangfrom {\hskip #3\relax 
    \@svsec}{\interlinepenalty \@M #8\par }\endgroup 
  \csname #1mark\endcsname {#7}
  \else 
  \def \@svsechd {#6\hskip #3\@svsec #8\csname #1mark\endcsname {#7}}
  \fi \@xsect {#5}}

\def\@@@startsection#1#2#3#4#5#6{%
  \if@noskipsec \leavevmode \fi \par \@tempskipa #4\relax \@afterindenttrue 
  \ifdim \@tempskipa <\z@ \@tempskipa -\@tempskipa \@afterindentfalse 
  \fi \if@nobreak \everypar {}\else \addpenalty {\@secpenalty }\addvspace 
  {\@tempskipa }\fi \@ifstar {\@ssect {#3}{#4}{#5}{#6}}{\@dblarg 
    {\@@@sect {#1}{#2}{#3}{#4}{#5}{#6}}}}

\def\theparagraph{\thesection.\arabic{paragraph}}
\def\aparagraph{\@@@startsection{paragraph}{2}{\z@ }%
  {1.75ex plus .2ex minus .15ex}{-1em}{\bf(\theparagraph) } }
\def\paragraph{\@@@startsection{paragraph}{2}{\z@ }%
  {1.75ex plus .2ex minus .15ex}{-1em}{}{\bf(\theparagraph)} }

\let\c@theorem\c@paragraph

\title{Equivariant class group. I.\\
Finite generation of the Picard and the class groups of
an invariant subring}
\author{M{\sc itsuyasu} H{\sc ashimoto}}
\date{\normalsize
  Graduate School of Mathematics, Nagoya University\\
  Chikusa-ku,  Nagoya 464--8602 JAPAN\\
  {\small \tt hasimoto@math.nagoya-u.ac.jp}}

\begin{document}

\maketitle

\begin{abstract}
The purpose of this paper is to define equivariant 
class group of a
locally Krull scheme (that is, a scheme which is locally a prime
spectrum of a Krull domain) with an action of a flat group scheme,
study its basic properties, and apply it to prove the finite
generation of the class group of an invariant subring.

In particular, we prove the following.

Let $k$ be a field, $G$ a smooth $k$-group scheme of finite type,
and $X$ a quasi-compact quasi-separated locally Krull $G$-scheme.
Assume that there is a $k$-scheme $Z$ of finite type and a
dominating $k$-morphism $Z\rightarrow X$.
Let $\varphi:X\rightarrow Y$ be a $G$-invariant morphism such that
$\O_Y\rightarrow (\varphi_*\O_X)^G$ is an isomorphism.
Then $Y$ is locally Krull.
If, moreover, $\Cl(X)$ is finitely generated, then $\Cl(G,X)$ and $\Cl(Y)$ are
also finitely generated, where $\Cl(G,X)$ is the equivariant class group.

In fact, $\Cl(Y)$ is a subquotient of $\Cl(G,X)$.
For actions of connected group schemes on affine schemes, 
there are similar results of Magid and Waterhouse,
but our result also holds for disconnected $G$.
The proof depends on a similar result on (equivariant) Picard groups.
\end{abstract}

\section{Introduction}
The purpose of this paper is to define equivariant 
class group of a
locally Krull scheme with an action of a flat group scheme,
study its basic properties, and apply it to prove the finite
generation of the class group of an invariant subring.

A locally Krull scheme is a scheme which is locally the prime spectrum of
a Krull domain.
For Krull domains, see \cite{CRT} and \cite{Fossum}.
As a Noetherian normal domain is a Krull domain, 
a normal scheme of finite type over a field (e.g., a normal variety) 
is a typical example of a 
(quasi-compact quasi-separated) locally Krull scheme.
Although a Krull domain is integrally closed, it may not be Noetherian.

Generalizing the theory of class groups of Noetherian normal domains, 
there is a well established theory of class groups of Krull domains
\cite{Fossum}.
In this paper, we also consider non-affine locally Krull schemes.
Also, we consider the equivariant version of the theory of class groups
over them.

Let $Y$ be a quasi-compact integral locally Krull scheme.
Then the class group $\Cl'(Y)$ of $Y$ is defined to be the
free abelian group $\Div(Y)$ generated by the set of integral
closed subschemes of codimension one, modulo the linear equivalence.
The second definition of the class group is given by the use of
rank-one reflexive modules.
For a Krull domain $R$, an $R$-module $M$ is said to be reflexive 
(or divisorial), if $M$ is a submodule of a finitely generated module,
and the canonical map $M\rightarrow M^{**}$ is an isomorphism,
where $(?)^*$ denotes the functor $\Hom_R(?,R)$.
An $\O_Y$-module $\M$ is said to be reflexive if $\M$ is quasi-coherent,
and for any affine open subscheme $U=\Spec A$ of $Y$ such that $A$
is a Krull domain, $\Gamma(U,\M)$ is a reflexive $A$-module,
where $\Gamma(U,?)$ denotes the section at $U$.
The set of isomorphism classes $\Cl(Y)$ of rank-one reflexive $\O_Y$-modules
is an additive group with the addition
\begin{equation}\label{addition.eq}
[\M]+[\N]=[(\M\otimes_{\O_Y}\N)^{**}],
\end{equation}
where $(?)^*=\uHom_{\O_Y}(?,\O_Y)$.
It is easy to see that the formation $D\mapsto \Cal O_Y(D)$ is an isomorphism
from $\Cl'(Y)$ to $\Cl(Y)$, as in the 
well-known case of normal varieties over a field \cite[Appendix to 
section~1]{Reid}.

If $Y$ is a quasi-compact locally Krull scheme, 
then $Y=Y_1\times\cdots\times Y_r$ with each $Y_i$ being quasi-compact
integral locally Krull, and we may define $\Cl'(Y)=\Cl'(Y_1)\times\cdots
\times \Cl'(Y_r)$.
Similarly for $\Cl(Y)$, and still we have $\Cl(Y)\cong\Cl'(Y)$.

In the rest of this introduction,
let $S$ be a scheme, $G$ a flat $S$-group scheme, and
$X$ a $G$-scheme (that is, an $S$-scheme with a $G$-action).

Let $X$ be locally Krull.
The first purpose of this paper is to define the equivariant class group
$\Cl(G,X)$ of $X$ and study its basic properties.

Generalizing the second definition above, we define
$\Cl(G,X)$ to be the set of isomorphism classes of quasi-coherent
$(G,\O_X)$-modules which are reflexive as $\O_X$-modules.
We prove that $\Cl(G,X)$ is an additive group with the addition
given by (\ref{addition.eq}).

We give a simplest example.
If $S=X=\Spec k$ with $k$ a field, 
and $G$ is an algebraic group over $k$, then
$\Cl(G,X)$ is nothing but the character group $\Cal X(G)$ of $G$.
That is, it is the abelian group of one-dimensional representations of $G$.

We do not try to redefine $\Cl(G,X)$ from the viewpoint
of the first definition (that of $\Cl'(Y)$).
So we do not consider $\Cl'(Y)$ in the sequel, 
and always mean 
the group of isomorphism classes of
rank-one reflexive sheaves
by the class group $\Cl(Y)$ of $Y$ for a locally Krull scheme $Y$,
see (\ref{equivariant-class.par}).

We prove that removing closed subsets of codimension two or more does not
change the equivariant class group (Lemma~\ref{codim-two-ref.thm}).
We also prove that if $\varphi:X\rightarrow Y$ is a principal $G$-bundle
with $X$ locally Krull, then 
$Y$ is also locally Krull, and 
the inverse image functor induces an isomorphism
$\varphi^*:\Cl(Y)\rightarrow \Cl(G,X)$ (Proposition~\ref{pfb-cl-isom.thm}).
This isomorphism gives a source of intuitive idea of the equivariant class
group --- it is the class group of the quotient space (or better, quotient
stack).
In the continuation of this paper, we give some variations of 
this isomorphism.

In general, the prime spectrum 
of an invariant subring may not be a good quotient.
However, we can prove that if $\varphi:X\rightarrow Y$ is a $G$-invariant 
morphism such that $X$ is quasi-compact quasi-separated locally Krull and
$\O_Y\rightarrow (\varphi_*\O_X)^G$ is an isomorphism,
then $Y$ is also locally Krull (Lemma~\ref{Y-Krull.thm}), 
and $\Cl(Y)$ is a subquotient of 
$\Cl(G,X)$ (Lemma~\ref{subquotient.thm}).

Using this lemma, we study the finite generation of the class group of $Y$.
This is the second purpose of this paper.
We prove the following.

\begin{trivlist}\item[\bf Theorem~\ref{main2.thm}]
Let $k$ be a field, $G$ a smooth $k$-group scheme of finite type,
and $X$ a quasi-compact quasi-separated locally Krull $G$-scheme.
Assume that there is a $k$-scheme $Z$ of finite type and a
dominating $k$-morphism $Z\rightarrow X$.
Let $\varphi:X\rightarrow Y$ be a $G$-invariant morphism such that
$\O_Y\rightarrow (\varphi_*\O_X)^G$ is an isomorphism.
Then $Y$ is locally Krull.
If, moreover, $\Cl(X)$ is finitely generated, then $\Cl(G,X)$ and $\Cl(Y)$ are
also finitely generated.
\end{trivlist}

Note that a normal $G$-scheme $X$ of finite type over $k$ is
automatically quasi-compact quasi-separated locally Krull, and
the identity map $Z:=X\rightarrow X$ is a dominating map, and so 
the assumptions of the theorem is satisfied, see Corollary~\ref{main2-cor.thm}.

In \cite{Magid}, Magid proved that if $R$ is a finitely generated 
normal
domain over the algebraically closed field $k$,
$G$ is a connected algebraic group acting rationally on $R$, and 
the class group $\Cl(R)$ of $R$ is a finitely generated abelian group,
then the class group $\Cl(R^G)$ of the ring of invariants $R^G$ is
also finitely generated.
After that, Waterhouse \cite{Waterhouse} proved a similar result
on an action of a connected affine group scheme on a
Krull domain over arbitrary base field.
Theorem~\ref{main2.thm} is not a generalization of Waterhouse's theorem.
We assume the existence of $Z\rightarrow X$ as above, and he
describes the relationship between $\Cl(X)$ and $\Cl(Y)$ precisely
\cite[Theorem~4]{Waterhouse}.
On the other hand, we treat disconnected groups, and non-affine groups and
schemes.
The action of finite groups is classical (see for example, 
\cite[Chapter~IV]{Fossum}), but the author does not know if the 
theorem for this case is in the literature, though it is not so 
difficult.

Note that in Theorem~\ref{main2.thm}, even if $X$ is a normal variety,
$Y$ may not be locally Noetherian (but is still locally Krull), 
as Nagata's counterexample \cite{Nagata} shows.
In fact, there are some operations on rings such that under which Krull domains
are closed, but Noetherian normal domains are not.
Let $R$ be a domain.
For a subfield $K$ of the field of quotients $Q(R)$ of $R$, 
consider $K\cap R$.
If $R$ is Krull, then so is $K\cap R$.
Even if $R$ is a polynomial ring (in finitely many variables) 
over a subfield $k$ of $K\cap R$, 
$K\cap R$ may not be Noetherian \cite{Nagata}.
For a domain $R$, consider 
a finite extension field $L$ of $Q(R)$.
Let $R'$ be the integral closure of $R$ in $L$.
If $R$ is a Krull domain, then so is $R'$.
If $R$ is Noetherian, then $R'$ is a Krull domain (Mori--Nagata theorem, 
see \cite[(4.10.5)]{SH}).
Even if $R$ is a (Noetherian) regular local ring, $R'$ may not be Noetherian.
Indeed, the ring $R$ and $L=Q(R[d])$ in 
\cite[Appendix, Example~5]{Nagata2} gives such an example (this is one of 
so-called bad Noetherian rings.
If $R$ is Japanese, then clearly $R'$ is Noetherian).
If $Z$ is an integral 
quasi-compact locally Krull scheme, then $\Gamma(Z,\O_Z)$ is
a Krull domain (Lemma~\ref{finite-direct-Krull.thm}).
In particular, for a normal projective variety $Y$ and its
Cartier divisors $D_1,\ldots,D_n$, the multi-section ring
\[
\bigoplus_{\lambda\in \Bbb Z^n}\Gamma(Y,\O_Y(\lambda_1D_1+\cdots+
\lambda_nD_n))t_1^{\lambda_1}\cdots 
t_n^{\lambda_n}
\]
is a Krull ring (see also \cite[Theorem~1.1 (1)]{EKW}), but
not always Noetherian \cite{Mukai}.

Thus locally Krull schemes arise in a natural way
in algebraic geometry and commutative algebra.
Despite of some technical difficulties, it would be worth discussing 
(equivariant) class groups in the framework of locally Krull schemes.

Returning to Theorem~\ref{main2.thm}, it is proved as follows.
As $\Cl(Y)$ is a subquotient of $\Cl(G,X)$, it suffices to show
that the kernel of the map $\alpha:\Cl(G,X)\rightarrow\Cl(X)$ is
finitely generated, where $\alpha$ is the map forgetting the $G$-action.

This problem is further reduced to a similar problem for Picard groups.
For a general $G$-scheme $X$ (not necessarily locally Krull), 
the equivariant Picard group $\Pic(G,X)$ is the set of isomorphism classes
of $G$-equivariant invertible sheaves on $X$.
The addition is given by $[\L]+[\L']=[\L\otimes_{\O_X}L']$.
So if $X$ is locally Krull, $\Pic(G,X)$ is a subgroup of $\Cl(G,X)$, 
and the kernel of the map $\rho:\Pic(G,X)\rightarrow \Pic(X)$ agrees with
$\Ker\alpha$ above.
So Theorem~\ref{main2.thm} follows from the following

\begin{trivlist}\item[\bf Theorem~\ref{main.thm}]
Let $k$ be a field, 
$G$ a smooth $k$-group scheme of finite type, and 
$X$ a reduced $G$-scheme which is quasi-compact and quasi-separated.
Assume that there is a $k$-scheme $Z$ of finite type and a 
dominating $k$-morphism $Z\rightarrow X$.
Then $H^1\alg(G,\O^\times)=\Ker(\rho:\Pic(G,X)\rightarrow\Pic(X))$ is
a finitely generated abelian group.
\end{trivlist}

Note that a reduced $k$-scheme of finite type $X$ is
automatically reduced, quasi-compact and quasi-separated,
admitting a dominating map from a finite-type scheme,
see Corollary~\ref{main-cor2.thm}.

The proof of this theorem utilizes the description of 
$H^1\alg(G,\O^\times)$ in \cite[Chapter~7]{Dolgachev}.
If $\varphi:X
\rightarrow Y$ is a $G$-invariant morphism such that $\O_Y\rightarrow
(\varphi_*\O_X)^G$ is an isomorphism, $\Pic(Y)$ is a subgroup of 
$\Pic(G,X)$ (Lemma~\ref{pic-injective.thm}).
So under the assumption of the theorem, if $\Pic(X)$ is finitely generated,
then $\Pic(G,X)$ and $\Pic(Y)$ are finitely generated
(Corollary~\ref{main-cor.thm}).

We also give some description on 
$H^i\alg(G,\O^\times)$ for $i\geq 2$ for connected $G$ 
(Proposition~\ref{connected-cohomology.thm}).

Section~2 is preliminaries on the notation and the terminologies.

Section~3 is dedicated to prove a five-term exact sequence involving
the map $\rho:\Pic(G,X)\rightarrow\Pic(X)^G$, where
$\Pic(X)^G$ is the kernel of the map $\Pic(X)\rightarrow \Pic(G\times X)$
given by $[\Cal L]\mapsto a^*[\Cal L]-p_2^*[\Cal L]$
($a:G\times X\rightarrow X$ is the action, and $p_2$ is the second 
projection), see Proposition~\ref{five-term.thm}.
The exact sequence also involves the ``algebraic $G$-cohomology group of 
$\O_X^\times$,'' denoted by $H^i\alg(G,\O^\times)$ for $i=1,2$, see
(\ref{group-cohomology.par}).

Although the author cannot find exactly the same exact sequence in the
literature,
it is more or less well-known.
The first three terms of the exact sequence is treated in 
\cite[Chapter~7]{Dolgachev} (the first four terms for the finite group
action is also treated there).
This exact sequence is important in discussing the kernel and the cokernel
of $\rho$.

In section~4, we prove Theorem~\ref{main.thm}.
We utilize the description $\Ker\rho\cong H^1\alg(G,\O^\times)$, and
reduce the problem to the action of a finite group scheme on a 
finite scheme.
We also give some relationship between $H^1\alg(G,\O^\times)$ and the
character group $\Cal X(G)$ in some special cases.
We also describe $H^i\alg(G,\O^\times)$ for higher $i$ for 
a connected group action.

Section~5 corresponds to the first purpose described above.
We define $\Cl(G,X)$ for $X$ locally Krull, and 
discuss some basics on (equivariant) class groups on locally Krull schemes.

In section~6, we prove Theorem~\ref{main2.thm}.

The author thanks 
Professor I.~Dolgachev,
Professor O.~Fujino,
Professor G.~Kemper,
Professor K.~Kurano,
Professor J.-i.~Nishimura,
and
Professor S.~Takagi
for valuable advice.

\section{Preliminaries}\label{prelimilaries}

\paragraph For a commutative ring $R$, $Q(R)$ denotes its total ring of
fractions.
That is, the localization $R_S$ of $R$, where $S$ is the set of 
nonzerodivisors of $R$.
In particular, if $R$ is an integral domain, $Q(R)$ is its field of 
fractions.

\paragraph In this paper, 
for a scheme $X$ and its subset $\Gamma$, the codimension
$\codim_X\Gamma$ of $\Gamma$ in $X$ is 
$\inf_{\gamma\in\Gamma}\dim \O_{X,\gamma}$ by definition (cf.~\cite[chapter~0,
(14.2.1)]{EGA-IV-1}).
The codimension of the empty set in $X$ is $\infty$.

\paragraph
Throughout this paper, let $S$ be a scheme.
For an $S$-group scheme $G$, a $G$-scheme means an $S$-scheme with
a (left) action of $G$.
We say that $f:X\rightarrow Y$ is a
$G$-morphism if $f$ is an $S$-morphism, $X$ and $Y$ are $G$-schemes, and
$f(gx)=gf(x)$ holds.
In this case, we also say that $X$ is a $(G,Y)$-scheme.
A $(G,Y)$-morphism $h:X\rightarrow X'$ is a morphism between $(G,Y)$-schemes
which is both a $G$-morphism and a $Y$-morphism.
We say that $f:X\rightarrow Y$ is a $G$-invariant morphism
if $f$ is a $G$-morphism and $G$ acts on $Y$ trivially.
If so, $f(gx)=f(x)$ holds.

\paragraph\label{fpqc.par}
A morphism of schemes $\varphi:X\rightarrow Y$ is fpqc if it is 
faithfully flat, and for any quasi-compact open subset $V$ of $Y$, 
there exists some quasi-compact open subset $U$ of $X$ such that 
$\varphi(U)=V$.
For basics on fpqc property, see \cite[(2.3.2)]{Vistoli}.

\paragraph Let $Y$ be a $G$-scheme on which $G$-acts trivially.
A $(G,Y)$-scheme $\varphi:X\rightarrow Y$ is said to be a trivial
$G$-bundle if $X$ is $(G,Y)$-isomorphic to the second projection
$p_2:G\times Y\rightarrow Y$.

\begin{definition}
  We say that $\varphi:X\rightarrow Y$ is a principal $G$-bundle
 (or a $G$-torsor) (with respect to the fpqc topology) if 
 it is $G$-invariant, and there exists some fpqc $S$-morphism $Y'\rightarrow Y$ 
 such that the base change $X'=Y'\times_Y X\rightarrow Y'$ is a 
 trivial $G$-bundle.
\end{definition}

\begin{lemma}[{\cite[(4.43)]{Vistoli}}]\label{pfb-equiv.thm}
A $G$-invariant morphism $\varphi:X\rightarrow Y$ is a principal $G$-bundle
if and only if there exists some fpqc morphism $Y'\rightarrow Y$ which 
factors through $\varphi$, and the map $\Phi:G\times X\rightarrow 
X\times_Y X$ given by $\Phi(g,x)=(gx,x)$ is an isomorphism.
\qed
\end{lemma}

\section{The fundamental five-term exact sequence}

\paragraph
Let $(\Cal C,\O_{\Cal C})$ be a ringed site.
An $\O_{\Cal C}$-module $\Cal L$ is said to be invertible if for any 
$c\in\Cal C$ there exists some covering $(c_\lambda\rightarrow c)$ of $c$
such that for each $\lambda$, $\Cal L|_{c_\lambda}\cong \O_{\Cal C}|_{c_\lambda}$.
The set of isomorphism classes of invertible sheaves
is denoted by $\Pic(\Cal C)$.
It is an (additive) abelian group by the operation $[\Cal L]+[\Cal M]
:=[\Cal L\otimes_{\Cal O_{\Cal C}}\Cal M]$.
$\Pic(\Cal C)$ is called the Picard group of $\Cal C$.

An $\O_{\Cal C}$-module $\M$ is said to be quasi-coherent if for any 
$c\in\Cal C$, there exists some covering $(c_\lambda\rightarrow c)$ of $c$
such that for each $\lambda$, there exists some exact sequence
of $\Cal O_{\Cal C}|_{c_\lambda}$-modules
\[
\Cal F_1\rightarrow \Cal F_0\rightarrow \M|_{c_\lambda}\rightarrow 0
\]
with $\Cal F_1$ and $\Cal F_0$ free (where a free sheaf means a
(possibly infinite) direct sum of $\Cal O_{\Cal C}|_{c_\lambda}$).
Obviously, an invertible sheaf is quasi-coherent.

\paragraph
Let $\Sh(\Cal C)$ and $\Ps(\Cal C)$ denote the categories of abelian 
sheaves and presheaves, respectively.
For $\M\in\Sh(\Cal C)$, the Ext-group $\Ext_{\Sh(\Cal C)}^i(a\Bbb Z,\M)$ is
denoted by $H^i(\Cal C,\M)$, where $\Bbb Z$ is the constant presheaf on 
$\Cal C$ and $a\Bbb Z$ its sheafification.
Similarly, for $\N\in\Ps(\Cal C)$, $\Ext_{\Ps(\Cal C)}^i(\Bbb Z,\N)$ is
denoted by $H^i_{\Ps}(\Cal C,\N)$.
Let $q:\Sh(\Cal C)\rightarrow \Ps(\Cal C)$ be the inclusion.
As it has the exact left adjoint (the sheafification $a$), it is left exact,
and preserves injectives.
Its right derived functor $(R^iq)(\Cal M)$ is denoted by $\uH^i(\Cal M)$.
As $\Hom_{\Sh(\Cal C)}(a\Bbb Z,?)=\Hom_{\Ps(\Cal C)}(\Bbb Z,?)\circ q$, 
a Grothendieck spectral sequence
\begin{equation}\label{GSS.eq}
E^{p,q}_2=H^{p}_{\Ps}(\Cal C,\uH^q(\Cal M))\Rightarrow H^{p+q}(\Cal C,\Cal M)
\end{equation}
is induced.

\paragraph
Let $\O^\times$ denote the presheaf of abelian group defined by
$\Gamma(c,\O^\times)=\Gamma(c,\O_{\Cal C})^\times$.
It is a sheaf.
The following is due to de~Jong and others \cite[(20.7.1)]{SP}.

\begin{lemma}\label{deJong.thm}
There is an isomorphism
$H^1(\Cal C,\O^\times)\cong \Pic(\Cal C)$.
\end{lemma}

\paragraph
Let $(\Delta)$ be the full subcategory of the category of ordered sets
whose object set $\ob((\Delta))$ is $\{[0],[1],[2],\ldots\}$,
where $[n]=\{0<1<\cdots<n\}$.
A simplicial $S$-scheme is a contravariant functor from $(\Delta)$ to
the category of $S$-schemes $\Sch/S$, by definition.

We denote the subcategory of $(\Delta)$ such that
the object set is the same, but the morphism is restricted to injective maps
by $(\Delta)\mon$.

Let $X_\bullet$ be a $((\Delta)^{\mon})\op$-diagram of $S$-schemes,
that is, a contravariant functor from 
$(\Delta)^{\mon}$ to $\Sch/S$.
Then there is a projective resolution 
\[
\Bbb L=
\cdots
\xrightarrow{\partial_2}
L_1\Bbb Z_1
\xrightarrow{\partial_1}
L_0\Bbb Z_0
\rightarrow
\Bbb Z 
\rightarrow 
0
\]
of the constant presheaf $\Bbb Z$ on the Zariski site $\Zar(X_\bullet)$ 
of $X_\bullet$, see \cite[(4.3)]{ETI}.
Where $(?)_i:\Sh(\Zar(X_\bullet))\rightarrow\Sh(\Zar(X_i))$ is the restriction
functor \cite[(4.5)]{ETI}, and $L_i$ its left adjoint (see \cite[(5.1)]{ETI}).
$\partial_i: L_i\Bbb Z_i\rightarrow L_{i-1}\Bbb Z_{i-1}$ is the alternating sum
$u_0-u_1+u_2-\cdots+(-1)^iu_i$, where 
$u_j$ corresponds to the $j$th inclusion map 
\[
\Bbb Z_i 
\rightarrow 
(L_{i-1}\Bbb Z_{i-1})_i=\bigoplus_{j=0}^id_j^*(\Bbb Z_{i-1})
=\bigoplus_{j=0}^i \Bbb Z_i
\]
under the adjoint isomorphism of the adjoint pair $(L_i,(?)_i)$.
The exactness of the complex is checked easily after restricting to
each dimension by $(?)_i$.
Indeed, the complex is nothing but
\[
\cdots
\rightarrow
\bigoplus_{\phi\in\Hom([1],[i])}\Bbb Z\cdot \phi
\rightarrow
\bigoplus_{\phi\in\Hom([0],[i])}\Bbb Z\cdot \phi
\rightarrow
\bigoplus_{\phi\in\Hom(\emptyset,[i])}\Bbb Z\cdot \phi
\rightarrow
0
\]
when it is evaluated at $(i,U)$.
This complex computes the reduced homology group of the $i$-simplex, so
it is exact.

\begin{lemma}
For any $\Cal N\in\Ps(\Zar(X_\bullet))$, 
$H^i_{\Ps}(\Zar(X_\bullet),\Cal N)$ is the $i$th cohomology group of the 
complex
\[
0
\rightarrow
\Gamma(X_0,\Cal N_0)
\xrightarrow{d_0-d_1}
\Gamma(X_1,\Cal N_1)
\xrightarrow{d_0-d_1-d_2}
\Gamma(X_2,\Cal N_2)
\rightarrow\cdots.
\]
\end{lemma}

\begin{proof}
Follows from the isomorphism
\[
H^i_{\Ps}(\Zar(X_\bullet),\Cal N)=
\Ext^i_{\Ps(\Zar(X_\bullet))}(\Bbb Z,\Cal N)
=
H^i(\Hom_{\Ps(\Zar(X_\bullet))}(\Bbb L,\Cal N)).
\]
\end{proof}

\paragraph\label{group-cohomology.par}
Let $S$ be a scheme, and $G$ an $S$-group scheme.
Let $X$ be a $G$-scheme.
We can associate a simplicial scheme $B_G(X)$ to $X$,
see \cite[(29.2)]{ETI}.
Its restriction to $(\Delta)\mon$ is denoted by $B_G'(X)$.

Consider $X_\bullet=B_G'(X)$.
For $\Cal N\in \Ps(G,X)=\Ps(\Zar(B_G'(X)))$, we denote
$H^i_{\Ps}(\Zar(B_G'(X)),\Cal N)$ by $H^i\alg(G,\Cal N)$.
It is the $i$th cohomology group of the complex 
$\Hom_{\Ps(\Zar(B_G'(X)))}
(\Bbb L,\Cal N)$:
\[
0\rightarrow \Gamma(X,\Cal N_0)
\xrightarrow{d_0-d_1}\Gamma(G\times X,\Cal N_1)
\xrightarrow{d_0-d_1+d_2}\Gamma(G\times G\times X,\Cal N_2)
\rightarrow\cdots,
\]
where 
\[
d_i(g_{n-1},\ldots,g_{0},x)=
\left\{
\begin{array}{ll}
(g_{n-1},\ldots,g_1,g_0x) & (i=0) \\
(g_{n-1},\ldots,g_ig_{i-1},\ldots,g_0,x) & (0<i<n) \\
(g_{n-2},\ldots,g_0,x) & (i=n)
\end{array}
\right..
\]
We denote the group of $i$-cocycles (resp.\ $i$-coboundaries) of the
complex by $Z^i\alg(G,\Cal N)$ (resp.\ $B^i\alg(G,\Cal N)$).

\paragraph
Let $X$ be as above.
Then we denote $\Pic(\Zar(B_G'(X)))$ by $\Pic(G,X)$,
and call it the $G$-equivariant Picard group of $X$.
By \cite[Lemma~9.4]{ETI}, the restriction
$\Pic(G,X)=\Pic(\Zar(B_G'(X)))\rightarrow\Pic(\Zar(B_G^M(X)))$ is
an isomorphism,
where $\Delta_M$ is the full subcategory of $(\Delta)\mon$ with
the object set $\{[0],[1],[2]\}$, and
$B_G^M(X)$ is the restriction of $B_G'(X)$ to $\Delta_M$.

\paragraph
A $(G,\O_X)$-module is a module sheaf over the ringed site
$\Zar(B_G^M(X))$ by definition.

Note that $\Pic(G,X)$ is the set of isomorphism classes of
quasi-coherent $(G,\O_X)$-modules which are invertible sheaves as
$\O_X$-modules.
The addition of $\Pic(G,X)$ is given by $[\L]+[\L']=[\L\otimes_{\O_X}\L']$.

\paragraph
If $X$ is a $G$-scheme, then there is an obvious homomorphism
$\rho:\Pic(G,X)\rightarrow \Pic(X)$, forgetting the $G$-action.
If $Y$ is an $S$-scheme with a trivial $G$-action, then 
$\tau:\Pic(Y)\rightarrow \Pic(G,Y)$ such that $\tau[\Cal L]=[\Cal L']$ 
is induced, where $\Cal L'$ is $\Cal L$ with the trivial $G$-action.
So $\rho\circ\tau=\id_{\Pic(Y)}$.
If $\varphi:X\rightarrow Y$ is a $G$-morphism, then
$\varphi^*:\Pic(G,Y)\rightarrow \Pic(G,X)$ given by $\varphi^*[\Cal L]
=[\varphi^*\Cal L]$ is induced.
By abuse of notation, 
the map without the $G$-action $\Pic(Y)\rightarrow \Pic(X)$ is also
denoted by $\varphi^*$.
Also, for a $G$-invariant morphism $\varphi:X\rightarrow Y$,
$\varphi^*\circ\tau:\Pic(Y)\rightarrow \Pic(G,X)$ is also denoted by
$\varphi^*$.

\begin{lemma}\label{pic-injective.thm}
Let $\varphi:X\rightarrow Y$ be a $G$-invariant morphism.
If $\Cal O_Y\rightarrow(\varphi_*\Cal O_X)^G$ is an isomorphism, then
$\varphi^*:\Pic(Y)\rightarrow \Pic(G,X)$　is injective.
\end{lemma}

\begin{proof}
Note that the canonical
map $\Cal L \rightarrow (\varphi_*\varphi^*\Cal L)^G$ is an isomorphism.
Indeed, to check this, 
as the question is local, we may assume that $\Cal L\cong\O_Y$.
But this case is nothing but the assumption itself.
So if $\varphi^*\Cal L\cong \O_X$, then
\[
\Cal L\cong (\varphi_*\varphi^*\Cal L)^G\cong(\varphi_*\Cal O_X)^G\cong\O_Y,
\]
and the assertion follows immediately.
\end{proof}

\paragraph
We denote the category of quasi-coherent $(G,\O_X)$-modules by
$\Qch(G,X)$.

\begin{lemma}\label{pfb-pic.thm}
Let $\varphi:X\rightarrow Y$ be a principal $G$-bundle.
Then $\varphi^*: \Qch(Y)\rightarrow\Qch(G,X)$ is an equivalence.
The induced map $\varphi^*:\Pic(Y)\rightarrow\Pic(G,X)$ given by
$\varphi^*[\Cal L]=[\varphi^*\Cal L]$ is an isomorphism of abelian groups.
\end{lemma}

\begin{proof}
\cite[(4.46)]{Vistoli} applied to the stack $\Cal F\rightarrow \Sch/S$ of
quasi-coherent sheaves, $\varphi^*:\Qch(Y)\rightarrow\Qch(G,X)$ is
an equivalence.
This shows that $\varphi^*:\Pic(Y)\rightarrow\Pic(G,X)$ is bijective.
\end{proof}

\begin{proposition}\label{five-term.thm}
There is an exact sequence
\begin{multline*}
0\rightarrow H^1\alg(G,\O^\times)\rightarrow 
\Pic(G,X)
\xrightarrow{\rho}
\Pic(X)^G
\rightarrow\\
H^2\alg(G,\O^\times)
\rightarrow
H^2(\Zar(B_G'(X)),\O^\times),
\end{multline*}
where 
\[
\Pic(X)^G=\{[\Cal L]\in\Pic(X)\mid a^*\Cal L\cong p_2^*\Cal L\},
\]
and $\rho$ is the map forgetting the $G$-action, as before.
\end{proposition}

\begin{proof}
Consider the spectral sequence
\[
E^{p,q}_2=H^p\alg(G,\uH^q(\O^\times))\Rightarrow H^{p+q}(\Zar(B_G'(X)),\O^\times)
\]
and its five-term exact sequence
\[
0\rightarrow E^{1,0}_2 \rightarrow E^1 \rightarrow E^{0,1}_2
\rightarrow E^{2,0}_2 \rightarrow E^2.
\]
The result follows from Lemma~\ref{deJong.thm} immediately.
\end{proof}

\section{Main result}

\paragraph
Let $k$ be a field, and $V$ and $W$ be $k$-vector spaces.
Let $\alpha$ be an element of $V\otimes_k W$.
Let $\Phi_V:V\otimes_k W\otimes_k W^*\rightarrow V$ and
$\Phi_W:V\otimes_k W\otimes_k V^*\rightarrow W$ be the
map given by $\Phi_V(v\otimes w\otimes w^*)=(w^*(w))v$
and $\Phi_W(v\otimes w\otimes v^*)=(v^*(v))w$, respectively.
Then $c_V(\alpha):=\{\Phi_V(\alpha\otimes w^*)\mid w^*\in W^*\}$ 
and $c_W(\alpha):=\{\Phi_W(\alpha\otimes v^*)\mid v^*\in V^*\}$
are
subspaces of $V$ and $W$, respectively.
If $\alpha=\sum_{i=1}^n v_i\otimes w_i$ with $v_i\in V$ and $w_i\in W$, 
then $c_V(\alpha)$ is a subspace of the $k$-span 
$\langle v_1,\ldots,v_n\rangle$
of $v_1,\ldots,v_n$.
If, moreover, $w_1,\ldots,w_n$ is linearly independent, 
$c_V(\alpha)$ agrees with $\langle v_1,\ldots,v_n\rangle$.
If $\alpha=\sum_{i=1}^m\sum_{j=1}^nc_{ij}v_i\otimes w_j$
with $v_1,\ldots,v_m$ and $w_1,\ldots,w_n$ linearly independent and
$c_{ij}\in k$, then $\dim c_V(\alpha)=\dim c_W(\alpha)=\rank (c_{ij})$.
Note that $\alpha=v\otimes w\neq 0$ for some $v\in V$ and $w\in W$
if and only if $\dim c_V(\alpha)=\dim c_W(\alpha)=1$, and if this is the
case, $v$ and $w$ are bases of the one-dimensional spaces $c_V(\alpha)$ 
and $c_W(\alpha)$, respectively.

From this observation, we have the following two lemmas easily.

\begin{lemma}\label{tensor.thm}
Let $k$ be a field, and $V$ and $W$ be $k$-vector spaces.
If $v,v'\in V$, $w,w'\in W$, and $v\otimes w=v'\otimes w'\neq 0$, then there 
exists some unique $c\in k^\times $ such that $v'=cv$ and
$w'=c^{-1}w$.
\qed
\end{lemma}

\begin{lemma}\label{tensor3.thm}
Let $k$ be a field, and $V$ and $W$ be $k$-vector spaces.
Let $k'$ be an extension field of $k$, and $V'=k'\otimes_k V$ and
$W'=k'\otimes_k W$.
Let $\alpha$ be an element of $V\otimes_k W$.
If $1\otimes \alpha\in k'\otimes_k(V\otimes_k W)\cong V'\otimes_{k'}W'$
is of the form $\mu'\otimes \nu'$ for some $\mu'\in V'$ and $\nu'\in W'$,
then there exist some $\mu\in V$ and $\nu\in W$ such that 
$\alpha=\mu\otimes\nu$.
\qed
\end{lemma}

\begin{lemma}\label{unit-constant.thm}
Let $k$ be a field, and $X$ be a reduced $k$-scheme.
Assume that there is a $k$-scheme $Z$ of finite type and a 
dominating $k$-morphism $Z\rightarrow X$.
Then there is a short exact sequence of the form
\[
1\rightarrow K^\times \xrightarrow{\iota}\Gamma(X,\O_X)^\times
\rightarrow \Bbb Z^r\rightarrow 0,
\]
where $K$ is the integral closure of $k$ in $\Gamma(X,\O_X)$, and
$\iota$ is the inclusion.
\end{lemma}

\begin{proof}
This is proved similarly to \cite[(4.12)]{Hashimoto}.
\end{proof}

\begin{lemma}\label{tensor2.thm}
Let $k$ be a field, and $X$ and $Y$ be quasi-compact quasi-separated
$k$-schemes.
Then the canonical map $k[X]\otimes_k k[Y]\rightarrow k[X\times Y]$ is
an isomorphism, where $k[X]=\Gamma(X,\O_X)$ and so on.
\end{lemma}

\begin{proof}
First, the case that both $X$ and $Y$ are affine is trivial.

Second, assume that $X$ is affine.
There is a finite affine open covering $Y=\bigcup_{i=1}^r Y_i$ of $Y$.
As each $Y_i\cap Y_j$ is again quasi-compact by the 
assumption of the quasi-separated property,
there is a finite affine open covering $Y_i\cap Y_j=\bigcup_k Y_{ijk}$.
Then there is a commutative diagram
\[
\xymatrix{
0 \ar[r] & k[X]\otimes_k k[Y] \ar[d] \ar[r] &
k[X]\otimes_k \prod_i k[Y_i] \ar[d] \ar[r] &
k[X]\otimes_k \prod_{i,j,k} k[Y_{ijk}] \ar[d] \\
0 \ar[r] & k[X\times Y] \ar[r] &
\prod_i k[X \times Y_i] \ar[r] &
\prod_{i,j,k} k[X\times Y_{ijk}]
}.
\]
By the first step and the five lemma, the left most vertical arrow
is an isomorphism.

Lastly, consider the general case.
Arguing as in the second step, and using the result of the second step,
we are done.
\end{proof}

In the rest of this section, we prove the following

\begin{theorem}\label{main.thm}
Let $k$ be a field, 
$G$ a smooth $k$-group scheme of finite type, and 
$X$ a reduced $G$-scheme which is quasi-compact and quasi-separated.
Assume that there is a $k$-scheme $Z$ of finite type and a 
dominating $k$-morphism $Z\rightarrow X$.
Then $H^1\alg(G,\O^\times)=\Ker(\rho:\Pic(G,X)\rightarrow\Pic(X))$ is
a finitely generated abelian group.
\end{theorem}

The proof is divided into several steps.

\begin{proof}
Step~1.\ 
The case that $G$ is a finite group, and
$X=\Spec B$ is also finite over $k$.

As $\Pic X$ is trivial, we have that $H^1\alg(G,\O^\times)\cong H^1(G,B^\times)
\cong\Pic(G,X)$.
Let $N$ be the kernel of $G\rightarrow \GL(B)$.

Step 1--1.
The case that $N$ is trivial.
Then we claim that 
the canonical map $\varphi:X=\Spec B\rightarrow 
Y=\Spec B^G$ is a principal $G$-bundle.
In order to check this, we may assume that $B^G$ is a field.
Then $G$ acts on the set of primitive idempotents of $B$ transitively.
So if $B=B_1\times\cdots\times B_r$ with each $B_i$ being a field,
then $r=[G:H]$, where $H$ is the stabilizer of the unit element $e_1$
of $B_1$.
It is also easy to check that $B^G=(B_1)^H$.
So $\dim_{B^G}B=r\dim_{B^G}B_1=r\#H=\#G$.

For $b\in B$, if $H$ is the stabilizer of $b$, then
$b$ is a root of a separable polynomial $\phi(t)=\prod_{\sigma\in G/H}(t-\sigma b)
$.
This shows that $\varphi$ is \'etale finite.
As $G$ is finite, it is also a geometric quotient.
So $\Phi:G\times X\rightarrow X\times_Y X$ given by $\Phi(g,x)=(gx,x)$ is
finite surjective.
As $X\times_Y X$ is reduced, $B\otimes_{B^G} B\rightarrow k[G]\otimes_k B$
is injective.
By dimension counting as vector spaces over $B^G$, we have that $\Phi$ is
an isomorphism as claimed.

By the claim and by Lemma~\ref{pfb-pic.thm}, $\Pic(G,B)\cong\Pic(B^G)=0$,
as desired.

Step 1--2.
The case that $N=G$.
That is, the case that $G$ acts on $B$ trivially.
If $B\cong B_1\times\cdots\times B_r$, then $\Pic(G,B)\cong\prod_i\Pic
(G,B_i)$.
So we may assume that $B$ is a field.
As $\Pic(G,B)\cong\Pic(B\otimes_kG,B)$, we may assume that $B=k$.
Then $\Pic(G,k)$ is nothing but the group $\Cal 
X(G)$ of the isomorphism classes of
one-dimensional representations of $G$.
As $G$ is finite, $\Cal X(G)$ is finite, as desired.

Step 1--3.
The case that $N$ is arbitrary.
By the exact sequence
\[
0\rightarrow E^{1,0}_2 \rightarrow E^1 \rightarrow E^{0,1}_2
\]
of the Lyndon--Hochschild--Serre spectral sequence
\[
E^{p,q}_2=H^p(G/N,H^q(N,B^\times))\Rightarrow H^{p+q}(G,B^\times),
\]
there is an exact sequence
\[
0\rightarrow H^1(G/N,B^\times)\rightarrow H^1(G,B^\times)\rightarrow
H^1(N,B^\times).
\]
Now the assertion follows from Step 1--1 and 1--2, immediately.

Step~2.
The case that $G$ is a finite group scheme, and
$X=\Spec B$ is also finite over $k$.
Then there is a finite Galois extension $k'$ of $k$ such that 
$\Omega:=k'\otimes_k G$ is a finite group.
That is to say, $\dim_k k[G]$ equals the number of $k'$-rational points
of $G$.
Thus $\Omega$ is identified with $\Hom_{k'\alg}(k'\otimes_k H,k')$.
Set $\Gamma:=\Gal(k'/k)$ to be the Galois group.
Note that $\Gamma$ acts on $k'\otimes_k H$ by $\gamma(\alpha\otimes h)
=(\gamma \alpha)\otimes h$.
$\Gamma$ acts on the group
$\Omega$ by $(\gamma \omega)(\alpha\otimes h)
=\alpha(\gamma(\omega(1\otimes h)))$.
In other words, $\gamma \omega=\gamma\circ \omega \circ \gamma^{-1}$.

Let $M$ be a $(G,B)$-module.
Plainly, $k'\otimes_k M$ is a $(k'\otimes_k G,k'\otimes_k B)$-module.
In other words, $(\Omega,k'\otimes_k B)$-module.
$\Omega$ acts on $k'\otimes_k B$ as $k'$-algebra automorphisms by
$\omega(\alpha\otimes b)=\sum_{(b)}\omega(\alpha\otimes b_{(1)})\otimes b_{(0)}$,
where we employ Sweedler's notation.
$\Gamma$ also acts on $k'\otimes_k M$ by $\gamma(\alpha\otimes m)
=(\gamma\alpha)\otimes m$.
$\Omega$ acts on $k'\otimes_k M$ by
$\omega(\alpha\otimes m)=\sum_{(m)}\omega(\alpha\otimes m_{(1)})\otimes m_{(0)}$.

It is easy to see that
\begin{multline*}
(\gamma\omega)(\alpha\otimes m)=
\sum_{(m)}(\gamma\omega)(\alpha\otimes m_{(1)})\otimes m_{(0)}\\
=\gamma(\sum_{(m)}\omega(\gamma^{-1}\alpha\otimes m_{(1)})\otimes m_{(0)})
=(\gamma\circ\omega\circ\gamma^{-1})(\alpha\otimes m).
\end{multline*}
Thus the actions of $\Gamma$ and $\Omega$ on $k'\otimes_k M$ together induce
a $k$-linear 
action of the semidirect product $\Theta:=
\Gamma \ltimes\Omega$ on $k'\otimes_k M$.
Similarly, $\Theta$ acts on $k'\otimes_k B$ by $k$-algebra
automorphisms.
We also think that $\Omega$ acts trivially on $k'$, and thus $\Theta$
acts on $k'$, $k$-linearly.
Now $k'\otimes_k M$ is a $(\Theta,k'\otimes_k B)$-module in the sense that
the action $k'\otimes_k B\otimes_k k'\otimes_k M\rightarrow k'\otimes_k M$ 
of $k'\otimes_k B$ on $k'\otimes_k M$ is $\Theta$-linear.
Thus $M\mapsto k'\otimes_k M$ is a functor from the category $\Mod(G,B)$ of
$(G,B)$-modules to the category $\Mod(\Theta,k'\otimes_k B)$ of
$(\Theta,k'\otimes_k B)$-modules (note that the base field is $k$, and 
not $k'$).

Now let $N$ be a $(\Theta,k'\otimes_k B)$-module.
Then $N^\Gamma$ is a $B$-module, since $(k'\otimes_k B)^\Gamma=B$.
As $N$ is also an $\Omega$-module, it is an $H$-comodule.
Note that the coaction
\[
\omega_N:N\rightarrow N\otimes_k H
\]
is $\Gamma$-linear, where $\Gamma$ acts on $N\otimes_k H$ by 
$\gamma(n\otimes h)=\gamma n\otimes h$.
Indeed, $\Omega$ acts on $N$ by $\omega n=\sum_{(n)}\omega(n_{(1)})n_{(0)}$
(here we identify
$\Omega=\Hom_{k'\alg}(k'\otimes_k H,k')\cong \Hom_{k\alg}(H,k')$).
As $\gamma((\gamma^{-1}\omega)(n))=\omega(\gamma n)$, 
\[
\sum_{(n)}(\omega(n_{(1)}))(\gamma (n_{(0)}))=
\sum_{(\gamma n)}\omega((\gamma n)_{(1)})(\gamma n)_{(0)}.
\]
As $\omega$ is arbitrary and $\Omega$ is a $k'$-basis of $\Hom_k(H,k')$,
it follows that
\[
\sum_{(n)}\gamma(n_{(0)})\otimes n_{(1)}
=\sum_{(\gamma n)}(\gamma n)_{(0)}\otimes (\gamma n)_{(1)}.
\]
That is, $\omega_N$ is $\Gamma$-linear.
So $N^\Gamma$ is an $H$-subcomodule of $N$.

As $B\otimes_k N \rightarrow N$ is $H$-linear,
$B\otimes_k N^\Gamma\rightarrow N^\Gamma$ is also $H$-linear, as can be
checked easily.
Thus $N^\Gamma$ is a $(G,B)$-module.

These functors $M\mapsto k'\otimes_k M$ and $N\mapsto N^\Gamma$ give
an equivalence.
Indeed, $k'\otimes_k X\rightarrow X$ is a principal $\Gamma$-bundle.
So the map $M\rightarrow(k'\otimes_k M)^\Gamma$ and $k'\otimes_k N^\Gamma
\rightarrow N$ are isomorphisms of $(\Gamma,k'\otimes_k B)$-modules and
$B$-modules, respectively.
We show that the map
$M\rightarrow(k'\otimes_k M)^\Gamma$ is also $G$-linear.
As $G$ acts on both $k$ and $k'$ trivially, the inclusion $k\hookrightarrow
k'$ is $G$-linear.
It follows that $M\rightarrow k'\otimes_k M$ is $G$-linear.
As $(k'\otimes_k M)^\Gamma$ is a $G$-submodule of $k'\otimes_k M$, 
the map $M\rightarrow(k'\otimes_k M)^\Gamma$ is $G$-linear.
Next, we show that $k'\otimes_k N^\Gamma\rightarrow N$ is $\Omega$-linear.
This is equivalent to say that it is $G$-linear.
As the map is the composite $k'\otimes_k N^\Gamma\hookrightarrow
k'\otimes_k N \rightarrow N$, this is trivial.

Thus we have an equivalence of categories 
$\Mod(G,X)\cong \Mod(\Theta,k'\otimes_k X)$, mapping $\Cal M$ to 
$p_2^*\Cal M$, where $p_2:k'\otimes_k X\rightarrow X$ is the canonical 
projection.
It is easy to see that $\Cal M$ is an invertible sheaf if and only if
$p_2^*\Cal M$ is.
Thus the equivalence induces an isomorphism $\Pic(G,X)\cong
\Pic(\Theta,k'\otimes_k X)$.
Thus changing $G$ to $\Theta$, $X$ to $k'\otimes_k X$, and without changing
the base field $k$, we may and shall assume that $G$ is a 
finite group.
But this case is done in Step~1.

Step~3.
The case that both $G=\Spec H$ and $X=\Spec B$ are affine.
Let $H_0$ and $B_0$ be the integral closures of $k$ in $H$ and $B$,
respectively.
Then, $H_0\otimes_k H_0\otimes_k\cdots\otimes_k H_0$ is the integral
closure of $k$ in
$H\otimes_k H\otimes_k\cdots\otimes_k H$.
To verify this, we may assume that $k$ is separably closed by
\cite[(6.14.4)]{EGA-IV-2}.
By \cite[(13.3)]{Borel}, connected components of $G$ are isomorphic 
each other.
So letting $G^\circ=\Spec H_1$ be the identity component of $G$, 
it suffices to show that $k$ is integrally closed in $H_1^{\otimes n}$.
But this is the consequence of the geometric integrality of $H_1$ 
\cite[(1.2)]{Borel}.
Similarly, the integral closure of $k$ in $H^{\otimes n}\otimes_k B$ is
$H_0^{\otimes n}\otimes_k B_0$.
To verify this,
we may assume that both $H_0$ and $B_0$ are fields.
Then $Q(H^{\otimes n})\otimes_k B_0$ is integrally closed in
$Q(H^{\otimes n})\otimes_k B$ by \cite[(6.14.4)]{EGA-IV-2}.
On the other hand, as $k\subset Q(H^{\otimes n})$ is a regular extension,
$B_0\subset Q(H^{\otimes n}\otimes_k B_0)$ is integrally closed.

As the image of the coproduct $\Delta(H_0)$ is contained in $H_0\otimes_k H_0$,
it is easy to see that $H_0$ is a subHopf algebra of $H$.
As $\omega_B(B_0)\subset B_0\otimes_k H_0$, $B_0$ is an $H_0$-comodule algebra
which is also an $H$-subcomodule algebra of $B$.
So when we set $G_0=\Spec H_0$ and $X_0=X$, then $G_0$ is a quotient group
scheme of $G$ (it is \'etale over $k$), $G_0$ acts on $X_0$, and the
diagram
\[
\xymatrix{
G\times X \ar[d] \ar[r]^a & X \ar[d] \\
G_0\times X_0 \ar[r]^a & X_0
}
\]
is commutative.

Let $\Mod(\Bbb Z)$ be the category of abelian groups, and
$\Cal F$ be its Serre subcategory consisting of finitely generated
abelian groups.
Set $\Cal A$ to be the quotient $\Mod(\Bbb Z)/\Cal F$.
Then by Lemma~\ref{unit-constant.thm}, 
$\Hom_{\Ps(\Zar(B_{G_0}(X)))}(\Bbb L,\Cal O_{X_0}^\times)$
and 
$\Hom_{\Ps(\Zar(B_{G}(X)))}(\Bbb L,\Cal O_{X}^\times)$
are isomorphic as complexes in $\Cal A$.
So the first cohomology of one is zero in $\Cal A$ if and only if
the first cohomology of the other is zero in $\Cal A$.
Thus replacing $G$ by $G_0$ and $X$ by $X_0$, we may assume that both
$G$ and $X$ are finite.
But this case is done in Step~2.

Step~4.
The general case.
The product $G\times G\rightarrow G$ induces 
$k[G]\rightarrow k[G\times G]\cong k[G]\otimes_k k[G]$ by 
Lemma~\ref{tensor2.thm}.
From this, it is easy to get the commutative Hopf algebra structure of $k[G]$.
Set $G_1=\Spec k[G]$.
Then the canonical map $G\rightarrow G_1$ is a homomorphism of group schemes.
Similarly, the action $G\times X\rightarrow X$ induces
$k[X]\rightarrow k[G\times X]\cong k[G]\otimes_k k[X]$.
This makes $k[X]$ a (left) $k[G]$-comodule algebra.
So letting $X_1=\Spec k[X]$, $G_1$ acts on $X_1$.
Now it is easy to see that
$\Hom_{\Ps(\Zar(B_{G}(X)))}(\Bbb L,\Cal O_{X})$, which looks like
\[
0\rightarrow k[X]^\times\rightarrow
k[G\times X]^\times\rightarrow
k[G\times G\times X]^\times\rightarrow\cdots
\]
agrees with
$\Hom_{\Ps(\Zar(B_{G_1}(X_1)))}(\Bbb L,\Cal O_{X_1})$.

So replacing $G$ by $G_1$ and $X$ by $X_1$, we may assume that
both $G$ and $X$ are affine.
But this case is done in Step~3.

This completes the proof of the theorem.
\end{proof}

As a reduced $k$-scheme of finite type is quasi-compact quasi-separated reduced
and is dominated by some $k$-scheme of finite type, we immediately have

\begin{corollary}\label{main-cor2.thm}
Let $k$ be a field, 
$G$ a smooth $k$-group scheme of finite type, and 
$X$ a reduced $G$-scheme of finite type.
Then $H^1\alg(G,\O^\times)=\Ker(\rho:\Pic(G,X)\rightarrow\Pic(X))$ is
a finitely generated abelian group.
\qed
\end{corollary}

\begin{corollary}\label{main-cor.thm}
Let $k$, $G$, $X$, and $Z\rightarrow X$ be as in Theorem~\ref{main.thm}.
Let $\varphi:X\rightarrow Y$ be a $G$-invariant morphism.
If $\O_Y\rightarrow (\varphi_*\O_X)^G$ is an isomorphism, then
the kernel of the map $\varphi^*:\Pic(Y)\rightarrow\Pic(X)$ is
a finitely generated abelian group.
\end{corollary}

\begin{proof}
Consider the commutative diagram
\[
\xymatrix{
0 \ar[r]& \Ker \rho \ar[r] & \Pic(G,X) \ar[r]^\rho & \Pic(X) \\
0 \ar[r]& \Ker \varphi^*\ar[r]& \Pic(Y) \ar[u]_{\varphi^*} \ar[r]^{\varphi^*}
& \Pic(X) \ar[u]_{\id}
}
\]
with exact rows.
Then by Lemma~\ref{pic-injective.thm}, 
the vertical arrow $\varphi^*:\Pic(Y)\rightarrow \Pic(G,X)$ is
an injective map, which maps $\Ker \varphi^*$ injectively into $\Ker \rho$.
As $\Ker\rho$ is finitely generated by the theorem, 
$\Ker\varphi^*$ is also finitely generated.
\end{proof}

\begin{lemma}
Let $G$ be a $k$-group scheme of finite type.
Then the character group
\[
\Cal X(G)=\{\chi\in k[G]^\times \mid \chi(g_1g_0)=\chi(g_1)\chi(g_0)\}
\]
is a finitely generated abelian group.
\end{lemma}

\begin{proof}
Extending $k$, we may assume that $k$ is algebraically closed.
As $\Cal X(G)=\Cal X(\Spec k[G])$, we may assume that $G$ is affine.
If $G$ is finite, then $G$ has only finitely many irreducible representations,
so $\Cal X(G)$ is also finite.
If $G$ is $\Bbb G_a$, then $k[\Bbb G_a]^\times=k^\times$, and so
$\Cal X(\Bbb G_a)$ is trivial.
If $G=\Bbb G_m$, then $\Cal X(G)\cong \Bbb Z$, as is well-known.
If $N$ is a closed normal subgroup of $G$, then
\[
0\rightarrow \Cal X(G/N)\rightarrow \Cal X(G)\rightarrow \Cal X(N)
\]
is exact.

Letting $N=G^\circ$ be the identity component of $G$, we may assume that
$G$ is either finite or connected.
The finite case is already done, so we consider the case that 
$G$ is connected.
Letting $N$ be the unipotent radical, we may assume that
$G$ is either reductive or unipotent.
If $G$ is unipotent, then $G$ has a normal subgroup $N$ which is isomorphic
to $\Bbb G_a$ and $G/N$ is still unipotent.
So this case is done by the induction on the dimension.
If $G$ is reductive, then $\Cal X(G)\cong \Cal X(G/[G,G])$, and
$G/[G,G]$ is a torus.
So we may assume that $G$ is a torus, and this case is also done by the
induction on the dimension.
\end{proof}

\begin{lemma}[cf.~{\cite[(1.8)]{Sweedler}}, {\cite[Theorem~2]{Rosenlicht}}]
\label{units.thm}
Let $k$ be a field, $X$ and $Y$ $k$-schemes such that 
$X$ is quasi-compact quasi-separated and $k[X]$ reduced, 
and $k$ is algebraically closed in $k[X]$.
Assume one of the following.
\begin{enumerate}
\item[\bf 1]
$Y$ is integral with the rational function field 
$\O_{Y,\eta}$ being a regular extension of $k$,
where $\eta$ is the generic point of $Y$.
\item[\bf 2] $Y$ is quasi-compact quasi-separated, $B=\Gamma(Y,\O_Y)$ is
a domain such that the quotient field $Q(B)$ is a regular extension of $k$.
\end{enumerate}
Then for any $\alpha\in\Gamma(X\times Y,\O_{X\times Y})^\times$,
there exist $\mu\in\Gamma(X,\O_X)^\times $ and 
$\nu\in\Gamma(Y,\O_Y)^\times$ such that
$\alpha(x,y)=\mu(x)\nu(y)$ for $x\in X$ and $y\in Y$.
\end{lemma}

\begin{proof}
We may and shall assume that $X$ is nonempty.

First consider the case that $Y=\Spec B$ is affine.
Then {\bf 1} and {\bf 2} say exactly the same thing.
By Lemma~\ref{tensor2.thm}, 
$\Gamma(X\times Y,\O_{X\times Y})=\Gamma(X,\O_X)\otimes_k B$.
Replacing $X$ by $\Spec \Gamma(X,\O_X)$, we may assume that
$X=\Spec A$ is affine.
There are finitely generated $k$-subalgebras $A_0$ of $A$ and $B_0$ of $B$ 
such that
$\alpha\in (A_0\otimes_k B_0)^\times$.
We are to prove that there exist some $\mu\in A$ and $\nu\in B$ such that
$\alpha=\mu\otimes\nu$.
Replacing $A$ by $A_0$ and $B$ by $B_0$, 
we may assume that $A$ and $B$ are finitely generated over $k$.
Let $k\sep$ be the separable closure of $k$.
By \cite[(19.1)]{SH}, $k\sep$ is normal over $k$.
Then by \cite[(6.14.4)]{EGA-IV-2}, $k\sep$ is integrally closed 
in $k\sep\otimes_k A$.
Clearly, $k\sep\otimes_k A$ is reduced and finitely generated over  $k\sep$.
Moreover, $k\sep\otimes_k B$ is a finitely generated domain over $k\sep$, 
and $Q(k\sep\otimes_k B)$ is a regular extension field over $k\sep$.
By Lemma~\ref{tensor3.thm}, replacing $k$ by its separable closure
$k\sep$, we may assume that $k$ is separably closed.
As $Y=\Spec B$ is geometrically integral over $k$, 
there is at least one $k$-algebra map $B\rightarrow k$ by
\cite[(AG.13.3)]{Borel}.

As in the proof of \cite[(1.8)]{Sweedler}, 
set $R=\bigotimes_{\frak U} B$, where $\frak U$ is an uncountable set.
Then $R$ is an integral domain, and its field of fractions $K$ is a 
regular extension of $k$.
By \cite[(19.1)]{SH}, $K$ is normal over $k$.
By \cite[(6.14.4)]{EGA-IV-2}, $K$ is integrally closed in $A\otimes_k K$.
By Lemma~\ref{unit-constant.thm}, 
$(A\otimes_k K)^\times/K^\times\cong \Bbb Z^n$
for some $n$.

Arguing as in \cite[(1.8)]{Sweedler}, we have that 
$\alpha\in (A\otimes_k B)^\times$ is of the form $\mu\otimes\nu$ 
for $\mu\in A^\times$ and $\nu\in B^\times$, as desired.

Next consider the general $Y$, and assume {\bf 1}.
Let $Q=\O_{Y,\eta}$.
Then there exist some $\mu\in\Gamma(X,\O_X)^\times$ and
$\nu\in Q^\times $ such that $\alpha=\mu\otimes \nu$ in 
$\Gamma(X\times Z,\O_{X\times Z})=\Gamma(X,\O_X)\otimes_k Q$, where
$Z=\Spec Q$.
Also, for an affine open subset $U=\Spec C$ of $Y$, 
there exist some $\mu'\in\Gamma(X,\O_X)^\times$ and
$\nu'\in C^\times$ such that $\alpha=\mu'\otimes\nu'$ in 
$\Gamma(X,\O_X)\otimes_k C$.
So $\mu'\otimes\nu'=\alpha=\mu\otimes\nu$ in 
$\Gamma(X,\O_X)\otimes_k Q$.
By Lemma~\ref{tensor.thm}, there exists some $c\in k^\times$
such that $\mu'=c\mu$ and $\nu'=c^{-1}\nu$.
This shows that $\nu,\nu^{-1}\in \bigcap_U \Gamma(U,\O_Y)=\Gamma(Y,\O_Y)$.
So $\nu\in\Gamma(Y,\O_Y)^\times$.
$\alpha(x,y)=\mu(x)\nu(y)$ holds, and this is what we wanted to prove.

The case {\bf 2} is reduced easily to the affine case, using
Lemma~\ref{tensor2.thm}.
\end{proof}

The following corollary for the case that $k$ is algebraically closed 
goes back to Rosenlicht \cite[Theorem~3]{Rosenlicht}.

\begin{corollary}\label{char.thm}
Let $k$ be a field, and $G$ a smooth connected $k$-group scheme of finite type.
If $\chi\in k[G]^\times$ and $\chi(e)=1$, where $e$ is the unit element, then
$\chi\in \Cal X(G)$.
\end{corollary}

\begin{proof}
We can write $\chi(g_1g_0)=\chi_1(g_1)\chi_0(g_0)$ with
$\chi_1(e)=\chi_0(e)=1$.
Then letting $g_1=e$ or $g_0=e$, we have $\chi_1=\chi_0=\chi$.
So $\chi\in\Cal X(G)$.
\end{proof}

\begin{lemma}\label{polynomial.thm}
Let $k$ be a field, and $Y$ a $k$-scheme.
Let $X$ be a quasi-compact quasi-separated $k$-scheme such that 
$k[X]$ is reduced.
Assume that either
\begin{enumerate}
\item[\bf 1] $\bar k\otimes_k Y$ is integral; or
\item[\bf 2] $\bar k\otimes_k k[Y]$ is integral, and 
$Y$ is quasi-compact quasi-separated,
\end{enumerate}
where $\bar k$ is the algebraic closure of $k$.
If the unit group of 
$\bar k\otimes_k k[Y]$ is ${\bar k}^\times$,
then $k[X]^\times\rightarrow k[X\times Y]^\times$ is an isomorphism.
\end{lemma}

\begin{proof}
Note that $X$ has only finitely many connected components 
$X_1,\ldots,X_r$.
Replacing $X$ by each $X_i$, we may assume that $X$ is connected.
It is easy to check that the integral closure $K$ of $k$ in $k[X]$ is
an algebraic extension field of $k$.
Applying Lemma~\ref{units.thm} to $K$ instead of $k$, and $K\otimes_k Y$ 
instead of $Y$, For any unit $\alpha\in k[X\times Y]^\times$,
there exists some $\mu\in K[X]^\times$ and $\nu\in K[K\otimes_k Y]^\times$
such that $\alpha(x,y)=\mu(x)\nu(y)$.
By assumption, $K[K\otimes_k Y]^\times=K^\times$, and hence
$k[X]^\times\rightarrow k[X\times Y]^\times$ is surjective.
Injectivity is easy, and we are done.
\end{proof}

\begin{lemma}
Let $k$ be a field, and $G$ a quasi-compact quasi-separated $k$-group scheme
such that $k[G]$ is geometrically reduced over $k$.
Let $X$ be a $G$-scheme.
Assume that $\bar k\otimes_k X$ is integral, or $X$ is quasi-compact
quasi-separated and $\bar k\otimes_k k[X]$ is integral.
If the unit group of $\bar k\otimes_k k[X]$ is ${\bar k}^\times$, then
$H^i\alg(G,\O_X^\times)\cong H^i\alg(G,k^\times)$.
In particular, $H^1\alg(G,\O_X^\times)\cong \Cal X(G)$.
\end{lemma}

\begin{proof}
By Lemma~\ref{polynomial.thm},
the map 
$k[G^n]^\times\rightarrow k[G^n\times X]^\times$
is an isomorphism.
The lemma follows.
For the last assertion, see the next lemma.
\end{proof}

\begin{lemma}[cf.~{\cite[(7.1)]{Dolgachev}}]
Let $k$ be a field, $G$ a $k$-group scheme, and
$X$ a $G$-scheme.
Assume that $k[G]^\times\rightarrow k[G\times X]^\times$ induced by the
first projection is an isomorphism.
Then $H^1\alg(G,\O^\times)=\Ker(\rho:\Pic(G,X)\rightarrow\Pic(X))$ is
isomorphic to $\Cal X(G)$.
\end{lemma}

\begin{proof}
Note that $H^1\alg(G,\O^\times)$ is $Z^1\alg(G,\O^\times)/B^1\alg(G,\O^\times)$
by (\ref{group-cohomology.par}),
where
\[
Z^1\alg(G,\O^\times)=\{\chi\in k[G\times X]^\times\mid 
\chi(g_1,g_0x)\chi(g_1g_0,x)^{-1}\chi(g_0,x)=1\}
\]
and
\[
B^1\alg(G,\O^\times)=\{\phi(gx)\phi(x)^{-1}\mid \phi\in k[X]^\times\}.
\]
Note that for $\chi\in k[G\times X]^\times$ can be
written as $\chi(g,x)=\chi_0(g)$ for a unique $\chi_0\in k[G]^\times$.
Then as the map induced by the 
projection $k[G\times G]\rightarrow k[G\times G\times X]$ is injective,
$\chi\in Z^1\alg(G,\O^\times)$ if and only if 
$\chi_0\in\Cal X(G)$.

On the other hand, as $k[G\times X]^\times=k[G]^\times$, 
$k[X]^\times=k[X]^\times\cap k[G]^\times=k^\times$.
So $B^1\alg(G,\O^\times)$ is trivial, and we are done.
\end{proof}

\begin{example}
If $G$ is a quasi-compact quasi-separated $k$-group scheme with $k[G]$ reduced,
acting on the affine $n$-space $X=\Bbb A^n$.
Then $H^1\alg(G,X)\cong \Pic(G,X)\cong\Cal X(G)$.
\end{example}

\begin{proposition}\label{connected-cohomology.thm}
Let $G$ be a connected smooth $k$-group scheme of finite type, 
and $X$ a quasi-compact quasi-separated $G$-scheme such that
$k[X]$ is reduced and $k$ is integrally closed in $k[X]$.
Then for any $n\geq 0$, 
any $\chi\in k[G^n\times X]^\times$ can be written as
\[
\chi(g_{n-1},\ldots,g_1,g_0,x)=\chi_{n-1}(g_{n-1})\cdots \chi_0(g_0)\alpha(x)
\]
with $\chi_{n-1},\ldots,\chi_0\in\Cal X(G)$ and $\alpha\in k[X]^\times$
uniquely.
Moreover, $Z^0\alg(G,\O_X^\times)=(k[X]^G)^\times$, $B^0\alg(G,\O_X^\times)=
\{1\}$, and
\begin{multline*}
Z^n\alg(G,\O_X^\times)
=
\{\chi\in k[G^n\times X]^\times\mid \forall g\in G,\,\forall
x\in X\, \alpha(gx)
=\chi_0(g)\alpha(x),\\
\chi_1=\chi_2,\ldots,\chi_{n-3}=\chi_{n-2},\;
\chi_{n-1}=1\}=B^n\alg(G,\O_X^\times)
\end{multline*}
if $n\geq 2$ is even.
\[
Z^n\alg(G,\O_X^\times)=
\{\chi\in k[G^n\times X]^\times\mid \alpha=\chi_1=\chi_3=\cdots=\chi_{n-2}=1
\}
\]
if $n$ is odd.
$B^n\alg(G,\O_X^\times)=Z^n\alg(G,\O_X^\times)$ if $n\geq 3$ is odd, and
\[
B^1\alg(G,\O_X^\times)
=\{\chi\in k[G\times X]^\times\mid \alpha=1,\;\chi_0\in\Cal X(G,X)\},
\]
where
\[
\Cal X(G,X):=\{\chi\in\Cal X(G)\mid \exists\alpha\in k[X]^\times\,
\forall g\in G\,x\in X\,\alpha(gx)=\chi(g)\alpha(x)\}.
\]
Thus
\[
H^n\alg(G,\O_X^\times)
=
\left\{
\begin{array}{ll}
(k[X]^G)^\times & (n=0) \\
\Cal X(G)/\Cal X(G,X) & (n=1) \\
0 & (n\geq 2)
\end{array}
\right..
\]
\end{proposition}

\begin{proof}
Let $\partial^n$ be the boundary map in the complex in 
(\ref{group-cohomology.par}).
Then $\partial^0(\alpha)(g_0,x)=\alpha(g_0x)\alpha(x)^{-1}$, 
\begin{multline*}
\partial^n(\chi)(g_n,\ldots,g_0,x)
=
\chi(g_n,\ldots,g_1,g_0x)\chi(g_n,\ldots,g_2g_1,g_0,x)
\cdots\\
\chi(g_ng_{n-1},g_{n-2},\ldots,g_0,x)
\chi(g_n,\ldots,g_1g_0,x)^{-1}
\chi(g_n,\ldots,g_3g_2,g_1,g_0,x)^{-1}\\
\cdots\chi(g_n,g_{n-1}g_{n-2},\ldots,g_0,x)^{-1}
\chi(g_{n-1},g_{n-2},\ldots,g_0,x)^{-1}\\
=(\alpha(g_0x)\alpha(x)^{-1}\chi_0(g_0)^{-1})(\chi_1(g_2)\chi_2(g_2)^{-1})\\
\cdots(\chi_{n-3}(g_{n-2})\chi_{n-2}(g_{n-2})^{-1})\chi_{n-1}(g_n)
\end{multline*}
if $n\geq 2$ is even, and
\begin{multline*}
\partial^n(\chi)(g_n,\ldots,g_0,x)
=
\chi(g_n,\ldots,g_1,g_0x)\chi(g_n,\ldots,g_2g_1,g_0,x)
\cdots\\
\chi(g_n,g_{n-1}g_{n-2},\ldots,g_0,x)\chi(g_{n-1},g_{n-2},\ldots,g_0,x)
\chi(g_n,\ldots,g_1g_0,x)^{-1}\\
\chi(g_n,\ldots,g_3g_2,g_1,g_0,x)^{-1}\cdots
\chi(g_ng_{n-1},g_{n-2},\ldots,g_0,x)^{-1}\\
=\alpha(g_0x)\chi_1(g_2g_1)\cdots\chi_{n-2}(g_{n-1}g_{n-2})
\end{multline*}
if $n$ is odd.
The results follow easily.
\end{proof}

\begin{corollary}[cf.~{\cite[Lemma~7.1]{Dolgachev}}]
Let $G$ be a connected smooth $k$-group scheme of finite type, and
$X$ a quasi-compact quasi-separated $G$-scheme such that $k[X]$ is reduced.
Then $H^n\alg(G,\O_X^\times)=0$ for $n\geq 2$.
In particular, $\rho:\Pic(G,X)\rightarrow \Pic(X)^G$ is surjective.
\end{corollary}

\begin{proof}
If $X$ is disconnected, then we can argue componentwise, and we may
assume that $X$ is connected.
Let $K$ be the integral closure of $k$ in $k[X]$.
Then $K$ is a field.
Replacing $k$ by $K$ and $G$ by $K\otimes_k G$, we may assume that 
$k$ is integrally closed in $k[X]$.
Now invoke Proposition~\ref{connected-cohomology.thm}.
\end{proof}

\section{Equivariant class group of a locally Krull scheme with
  a group action}

\paragraph
Let $R$ be an integral domain with $K=Q(R)$.
An $R$-module $M$ is a {\em lattice} or {\em $R$-lattice} if 
$M$ is torsion-free and
$M$ is isomorphic to 
an $R$-submodule of a finitely generated $R$-module.
By definition, a finitely generated torsion-free $R$-module is a lattice.
A submodule of a lattice is a lattice.
The direct sum of two lattices is a lattice.

\begin{lemma}
  Let $M$ be an $R$-module.
  Then the following are equivalent.
  \begin{enumerate}
  \item[\bf 1] $M$ is a lattice.
  \item[\bf 2] There is a finitely generated $R$-free module $F$ and
    an injective $R$-linear map $M\hookrightarrow F$ and $a\in R\setminus
    0$ such that $aF\subset M$.
  \end{enumerate}
\end{lemma}

\begin{proof}
  {\bf 1$\Rightarrow$2}.
  By assumption, there is a finitely generated $R$-module $N$ and 
  an injection $M\rightarrow N$.
  Replacing $N$ by $N/N\tor$ if necessary,
  we may assume that $N$ is torsion-free, where $N\tor$ is the
  torsion part of $N$.
  Take $m_1,\ldots,m_r\in M$ which form a $K$-basis of $K\otimes_R M$.
  Take $n_{r+1},\ldots,n_s\in N$ such that 
  $m_1,\ldots,m_r,n_{r+1},\ldots,n_s$ is a $K$-basis of $K\otimes_R N$.
  Let $F_0$ and $G_0$ be $R$-spans of 
  $m_1,\ldots,m_r$ and $m_1,\ldots,m_r,n_{r+1},\ldots,n_s$, respectively.
  As $N$ is finitely generated, there exists some $a\in R\setminus 0$
  such that $N\subset a^{-1}G_0$.
  Then $F_0\subset M\subset N\cap M\subset a^{-1}G_0\cap(K\otimes_R F_0)
  =a^{-1}F_0$.
  Now set $F:=a^{-1}F_0$, and we are done.

  {\bf 2$\Rightarrow$1} is trivial.
\end{proof}

\begin{lemma}\label{basic2.thm}
  Let $M$ be an $R$-module.
  \begin{enumerate}
  \item[\bf 1] If $M$ is torsion-free (resp.\ a lattice) and 
    $R'$ a flat $R$-algebra which is a domain,
    then $M'=R'\otimes_R M$ is a torsion-free
    $R'$-module (resp. an $R'$-lattice).
  \item[\bf 2] Let $A_1,\ldots,A_r$ be $R$-algebras which are domains.
If $R\rightarrow \prod_i A_i$ is faithfully flat and each $A_i\otimes_R M$
is torsion-free as an $A_i$-module, then
    $M$ is torsion-free.
\item[\bf 3] Let $\Spec R=\bigcup_{i\in I}\Spec A_i$ be an affine open
covering, and assume that each $A_i\otimes_R M$ is a lattice.
Then $M$ is a lattice.
  \end{enumerate}
\end{lemma}

\begin{proof}
  {\bf 1} If $M$ is torsion-free, then $M\rightarrow K\otimes_R M$ is
  injective.
  By flatness, $M'\rightarrow R'\otimes_R K\otimes_R M$ is injective.
  As $K\otimes_R M$ is a $K$-free module, $R'\otimes_R K\otimes_R M$ is
  an $R'\otimes_R K$-free module.
  Hence the localization $R'\otimes_R K\otimes_R M\rightarrow Q(R')\otimes_R M
  =Q(R')\otimes_{R'}M'$ is injective.
  Thus $M'$ is torsion-free.

  If $M$ is a lattice and $M\subset N$ with $N$ being $R$-finite, then
  $M'\subset N'$ with $N'$ being $R'$-finite, and $M'$ is an $R'$-lattice.

  {\bf 2}
Let $K$ and $L_i$ be the field of fractions of $R$ and $A_i$, respectively.
  Then the diagram
  \[
  \xymatrix{
    M \ar[r]^j \ar[d]^{\delta} & K\otimes_R M \ar[d]^{\Delta} \\
    \bigoplus_i A_i\otimes_R M \ar[r]^{r} &
    \bigoplus_i L_i\otimes_R M
  }
  \]
  is commutative.
  As a faithfully flat algebra is pure \cite[Theorem~7.5, (i)]{CRT},
  $\delta$ is injective.
  If each $A_i\otimes_R M$ is torsion-free, then
  $r$ is injective, and hence $j$ is injective, and
  $M$ is torsion-free.

{\bf 3}
There exist $m_{i1},\ldots,m_{ir_i}\in K\otimes_R M$ 
such that
  the $A_i$-span of $m_{i1},\ldots, m_{ir_i}$ contains $A_i\otimes_R M$.
  Let $N$ be the $R$-submodule spanned by the all $m_{ij}$.
  Set $V=(N+M)/N$.
  Then $A_i\otimes_R V=0$ for any $i$.
  As $\Spec R=\bigcup_i \Spec A_i$ is an open covering, 
  we have that $V=0$.
  Hence $N\supset M$, and $M$ is a lattice.
\end{proof}

For an $R$-module $M$, set $M\tf:=M/M\tor$, where $M\tor$ is the
torsion part of $M$.

\begin{lemma}\label{Hom-lattice.thm}
  Let $M$ be an $R$-module such that $M\tf$ is isomorphic to a submodule
  of a finitely generated module.
  Let $N$ be a lattice.
  Then $\Hom_R(M,N)$ is a lattice.
\end{lemma}

\begin{proof}
  Replacing $M$ by $M\tf$, we may assume that $M$ is a lattice.
  Let $F$ be a finitely generated free $R$-module containing $N$.
  Then $\Hom_R(M,N)$ is a submodule of $\Hom_R(M,F)$.
  Replacing $N$ by $F$, we may assume that $N$ is finite free.
  As $\Hom_R(M,F)$ is a finite direct sum of $\Hom_R(M,R)$, we may assume
  that $N=R$.

  Take a finite free $R$-module $P$ and $a\in R\setminus 0$ such that
  $aP\subset M\subset P$.
  Then $a:P\rightarrow P$ induces a map $h:P\rightarrow M$ such that
  $C=\Coker h$ is annihilated by $a$.
  Then, dualizing, we get an injective map $M^*\rightarrow P^*$, 
  since $C^*=0$.
  Thus $M^*=\Hom_R(M,R)$ is a lattice, as desired.
\end{proof}

\begin{lemma}\label{tensor-lattice.thm}
  Let $M$ and $N$ be $R$-modules.
  Assume that $M\tf$ and $N\tf$ are lattices.
  Then $(M\otimes_R N)\tf$ is a lattice.
\end{lemma}

\begin{proof}
  The images of $M\tor\otimes_R N$ and $M\otimes_R N\tor$ in $M\otimes_R N$
  are torsion modules.
  So replacing $M$ and $N$ by $M\tf$ and $N\tf$, we may assume that
  $M$ and $N$ are lattices.
  Take finite free $R$-modules $F$ and $P$ and $a,b\in R$ such that
  $aF\subset M\subset F$ and $bP\subset N\subset P$.
  Set $K$ to be the kernel of $M\otimes_R N\rightarrow F\otimes_R P$.
  Then $K_{ab}$ is zero.
  So $K$ is a torsion module, and hence $(M\otimes_R N)\tf=(M\otimes_R N)/K$ is
  a submodule of $F\otimes_R P$.
\end{proof}

\paragraph
We say that an $R$-module $M$ is reflexive (or divisorial) 
if $M$ is a lattice, and
the canonical map $M\rightarrow M^{**}$ is an isomorphism,
see \cite{Fossum}.

\begin{lemma}\label{lattice-flat.thm}
  Let $R$ be a Krull domain, $M$ an $R$-lattice, $F$ and $P$
  flat $R$-modules.
  Then the canonical map
  \[
  \Hom_R(M,P)\otimes_R F\rightarrow \Hom_R(M,P\otimes_R F)
  \]
  is an isomorphism.
\end{lemma}

\begin{proof}
  It suffices to show that the two maps
  \[
  \Hom_R(M,R)\otimes_R(P\otimes_R F)\rightarrow
  \Hom_R(M,P\otimes_R F)
  \]
  and
  \[
  \Hom_R(M,R)\otimes_R P\rightarrow \Hom_R(M,P)
  \]
  are isomorphisms.
  So we may assume that $P=R$.

  Take a finitely generated $R$-free module $F'$ and $a\in R\setminus 0$
  such that $aF'\subset M\subset F'$.
  Let $\Cal P$ be the set of minimal primes of $Ra$.
  Then as submodules of $\Hom_K(K\otimes_R M,K\otimes_R F)$, 
  \begin{multline*}
    \Hom_R(M,R)\otimes_R F
    =
    \Hom_R(M,R[1/a]\cap\bigcap_{P\in\Cal P}R_P)\otimes_R F
    =\\
    (\Hom_R(M,R[1/a])\cap\bigcap_P\Hom_R(M,R_P))\otimes_R F
    =\\
    (\Hom_R(M,R[1/a])\otimes_R F)\cap\bigcap_P(\Hom_R(M,R_P)\otimes_R F)
    =\\
    \Hom_{R[1/a]}(R[1/a]\otimes_R M,R[1/a])\otimes_{R[1/a]}(R[1/a]\otimes_R
    F)\cap\\
    \bigcap_P(
    \Hom_{R_P}(M_P,R_P)\otimes_{R_P}F_P)
    =\\
    \Hom_{R[1/a]}(R[1/a]\otimes_R M,R[1/a]\otimes_R F)
    \cap
    \bigcap_P \Hom_{R_P}(M_P,F_P) 
    =\\
    \Hom_R(M,R[1/a]\otimes_R F)\cap
    \bigcap_P \Hom_R(M,R_P\otimes_R F)
    =\\
    \Hom_R(M,(R[1/a]\otimes_R F)\cap\bigcap_P (R_P\otimes_R F))
    =\\
    \Hom_R(M,(R[1/a]\cap\bigcap_P R_P)\otimes_R F)
    =\Hom_R(M,R\otimes_R F)=\Hom_R(M,F),
  \end{multline*}
  since $R[1/a]\otimes_R M$ and $M_P$ are finite free modules over 
  $R[1/a]$ and $R_P$, respectively.
\end{proof}

\begin{lemma}\label{Krull-descent.thm}
Let $\varphi:A\rightarrow B$ be a faithfully flat ring homomorphism,
and assume that $B$ is a finite direct product of \(Krull\) domains.
Then $A$ is a finite direct product of \(Krull\) domains.
\end{lemma}

\begin{proof}
Assume that $B$ is a finite direct product of domains.
As $B$ has only finitely many minimal primes, $A$ has finitely many 
minimal primes $P_1,\ldots,P_r$.
If $i\neq j$, then $P_i+P_j=A$.
Indeed, if not, $P_i+P_j\subset\frak m$ for some maximal ideal $\frak m$ of $A$.
Then, there is a prime ideal $M$ of $B$ such that $M\cap A=\frak m$.
As $B_M$ is a domain and $A_{\frak m}$ is its subring, $A_{\frak m}$ is a domain.
But this contradicts the assumption that $P_iA_{\frak m}$ and $P_jA_{\frak m}$ 
are different minimal primes of $A_{\frak m}$.
Thus $A$ is a direct product of integral domains.

Now we assume that $B$ is a finite direct product of Krull domains.
Then $A$ is a finite direct product of domains.
By localizing, we may assume that $A$ is a domain.
If $b/a\in B\cap Q(A)$ with $a,b\in A$, then $b\in aB\cap A=aA$.
So $b/a\in A$, and we have that $B\cap Q(A)=A$ in $Q(B)$.
The rest is easy.
\end{proof}

\begin{lemma}\label{flat-reflexive.thm}
  Let $R$ be a Krull domain, and $M$ be an $R$-module.
If $M$ is reflexive and
    $R'$ is a flat $R$-algebra which is a domain, then $M\otimes_R R'$ is
    reflexive.
\end{lemma}

\begin{proof}
By Lemma~\ref{basic2.thm}, $M\otimes_R R'$ is a lattice.
  We have isomorphisms
  \begin{multline*}
    \Hom_R(\Hom_R(M,R),R)\otimes_R R'\cong
    \Hom_{R'}(\Hom_R(M,R)\otimes_R R',R')\cong\\
    \Hom_{R'}(\Hom_{R'}(M\otimes_R R',R'),R').
  \end{multline*}

  Let $\frak D:M\rightarrow M^{**}=\Hom_R(\Hom_R(M,R),R)$ 
  be the canonical map.
  Then
  \[
  M\otimes_R R'\xrightarrow{\frak D\otimes_R 1}
  M^{**}\otimes_R R'
  \]
  is an isomorphism if and only if
  \[
  \frak D: M\otimes_R R'\rightarrow 
  \Hom_{R'}(\Hom_{R'}(M\otimes_R R',R'),R')
  \]
  is an isomorphism, see \cite[Lemma~2.7]{HO}.
\end{proof}

\begin{lemma}[{\cite[Corollary~5.5]{Fossum}}]\label{intersection.thm}
  Let $R$ be a Krull domain with $K=Q(R)$, and $M$ an $R$-lattice.
  As submodules of $K\otimes_R M=\Hom_K(\Hom_K(K\otimes_R M,K),K)$, 
  we have $M^{**}=\bigcap_{P\in X^1(R)}M_P$, where
  $X^1(R)$ is the set of height one primes of $R$.
  In particular, the following are equivalent.
  \begin{enumerate}
  \item[\bf 1] $M$ is reflexive;
  \item[\bf 2] $M=\bigcap_{P\in X^1(R)}M_P$ in $K\otimes_R M$.
  \end{enumerate}
\end{lemma}

\begin{proof}
  As submodules of $K\otimes_R M=\Hom_K(\Hom_K(K\otimes_R M,K),K)$, 
  \begin{multline*}
    \Hom_R(\Hom_R(M,R),R)=
    \Hom_R(\Hom_R(M,R),\bigcap_P R_P)
    =\\
    \bigcap_P \Hom_{R}(\Hom_R(M,R),R_P)
    =
    \bigcap_P \Hom_{R_P}(\Hom_R(M,R)_P,R_P)
    =\\
    \bigcap_P \Hom_{R_P}(\Hom_{R_P}(M_P,R_P),R_P)
    =
    \bigcap_P M_P.
  \end{multline*}
  The assertions follow.
\end{proof}

\begin{corollary}
Let $R$ be a Krull domain, and 
\[
0\rightarrow L\rightarrow M\rightarrow N
\]
be an exact sequence of $R$-lattices.
Then
\[
0\rightarrow L^{**}\rightarrow M^{**}\rightarrow N^{**}
\]
is also exact.
\end{corollary}

\begin{proof}
This is because
\[
0\rightarrow
\bigcap_{P\in X^1(R)}L_P
\rightarrow
\bigcap_{P\in X^1(R)}M_P
\rightarrow
\bigcap_{P\in X^1(R)}N_P
\]
is exact.
\end{proof}

\begin{corollary}\label{second-syzygy.thm}
Let $R$ be a Krull domain, and
\[
0\rightarrow L\rightarrow M\rightarrow N
\]
be an exact sequence of $R$-modules.
If $M$ is reflexive and $N$ is torsion-free, then $L$ is reflexive.
\end{corollary}

\begin{proof}
Being a submodule of the lattice $M$, we have that $L$ is a lattice.
Now apply the five lemma to the diagram
\[
\xymatrix{
0 \ar[r] &
\bigcap_{P\in X^1(R)}L_P
\ar[r] &
\bigcap_{P\in X^1(R)}M_P
\ar[r] &
\bigcap_{P\in X^1(R)}N_P\\
0 \ar[r] &
L \ar[r] \ar[u] &
M \ar[r] \ar[u] &
N \ar[u]
}.
\]
\end{proof}

\begin{lemma}\label{height-one.thm}
Let $R$ be an integral domain.
Let $R'$ be a faithfully flat $R$-algebra which is also a finite
direct product of Krull domains.
If $\frak p$ is a height-one prime ideal of $R$,
then there exists some height-one prime ideal $P$ of $R'$ such that
$P\cap R=\frak p$.
\end{lemma}

\begin{proof}
$R$ is a Krull domain by Lemma~\ref{Krull-descent.thm}.
By localizing, we may assume that $R$ is a DVR.
Let $\pi$ be a generator of the maximal ideal $\frak p$ of $R$.
As $\pi R'\neq R'$ by the faithful flatness, there exists some minimal
prime $P$ of $\pi R'$.
Then $P$ is of height one, since $R'$ is a finite direct product of
Krull domains.
The assertion follows.
\end{proof}

\begin{lemma}\label{reflexive-descent.thm}
Let $R$ be an integral domain, and $M$ an $R$-module.
Let $A_1,\ldots,A_r$ be $R$-algebras which are Krull domains 
such that
$R'=\prod_{i=1}^r A_i$ is a faithfully flat $R$-algebra.
If each $A_i\otimes_R M$ is a lattice \(resp.\ reflexive\), then 
$M$ is a lattice \(resp.\  reflexive\).
\end{lemma}

\begin{proof}
Note that $R$ is a Krull domain by Lemma~\ref{Krull-descent.thm}.

Assume that each $A_i\otimes_R M$ is a lattice.
Then $M$ is torsion-free by Lemma~\ref{basic2.thm}.
Obviously, $K\otimes_R M$ is a finite dimensional $K$-vector space.

Let $F$ be any finite free $R$-submodule of $K\otimes_R M$ such that 
$K\otimes_R F=K\otimes_R M$.
Set $R'=\prod_i A_i$.
Then in $Q(R')\otimes_{R}M$, there exists some nonzerodivisor $a$ of $R'$
such that $R'\otimes_R M \subset a^{-1}(R'\otimes_R F)$.
Let $P_1,\ldots,P_s$ be the complete list of height one primes of $R'$ such 
that $a\in P_i$.
Set $\frak p_i:=P_i\cap R$.
For each height one prime $\frak p$ of $R$, choose height one prime ideal
$P(\frak p)$ of $R'$ such that $P(\frak p)\cap R=\frak p$ (we can do so
by Lemma~\ref{height-one.thm}).
Let $v_{\frak p}$ be the normalized discrete valuation of $Q(R'_{P(\frak p)})$
corresponding to $R'_{P(\frak p)}$,
and $n_{\frak p}$ be the ramification index.
That is, $\frak pR'_{P(\frak p)}=(P(\frak p)R'_{P(\frak p)})^{n_{\frak p}}$.

Take $b\in R\setminus 0$ such that $v_{\frak p}(b)\geq
v_{\frak p}(a)$ for any $\frak p$.
This is possible, since $v_{\frak p}(a)=0$ unless $\frak p=\frak p_i$ for
some $i$.
Then for any $\frak p$,
\begin{multline*}
M\subset (R'_{P(\frak p)}\otimes_R M)\cap (K\otimes_R M)
\subset a^{-1}(R'_{P(\frak p)}\otimes_R F)\cap (K\otimes_R F)
=\\
(a^{-1}R'_{P(\frak p)}\cap K)\otimes_R F
\subset (\frak p R_{\frak p})^{-\ru{v_{\frak p}(a)/n_{\frak p}}}\otimes_R F
\subset b^{-1}R_{\frak p}\otimes_R F.
\end{multline*}
Thus
\[
M\subset \bigcap_{\frak p}b^{-1}(R_{\frak p}\otimes_R F)=b^{-1}F.
\]
This shows that $M$ is a lattice.

Next assume that $A_i\otimes_R M$ is reflexive for any $i$.
Then
$\frak D\otimes 1_{A_i}:M\otimes_R A_i\rightarrow M^{**}\otimes_R
  A_i$ is an isomorphism for any $i$.
  So $\frak D:M\rightarrow M^{**}$ is an isomorphism.
\end{proof}

\begin{lemma}\label{Krull-Hom.thm}
  Let $R$ be a Krull domain, $M$ an $R$-lattice, $N$ a reflexive 
  $R$-module, and $F$ and $P$
  flat $R$-modules.
  Then the canonical map
  \[
  \Hom_R(M,N\otimes_R P)\otimes_R F\rightarrow \Hom_R(M,N\otimes_R P\otimes_R F)
  \]
  is an isomorphism.
\end{lemma}

\begin{proof}
  Similar to Lemma~\ref{lattice-flat.thm}.
  Use Lemma~\ref{intersection.thm}.
\end{proof}

\begin{lemma}\label{Hom-reflexive.thm}
  Let $R$ be a Krull domain,
  $M$ an $R$-module such that $M\tf$ is a lattice,
  and $N$ a reflexive $R$-module.
  Then $\Hom_R(M,N)$ is reflexive.
\end{lemma}

\begin{proof}
  We may assume that $M$ is a lattice.
  By Lemma~\ref{Hom-lattice.thm}, $\Hom_R(M,N)$ is an $R$-lattice.
  By Lemma~\ref{Krull-Hom.thm},
  \begin{multline*}
    \Hom_R(M,N)=\Hom_R(M,\bigcap_{P\in X^1(R)}N_P)
    =\bigcap_P\Hom_{R}(M,N_P)\\
    =\bigcap_P\Hom_R(M,N)_P.
  \end{multline*}
\end{proof}

\begin{lemma}\label{associativity.thm}
  Let $R$ be a Krull domain, and $M$ and $N$ be $R$-modules such that
  $M\tf$ and $N\tf$ are lattices.
  Then the canonical map
  \[
  (M\otimes_R N)^{**}\rightarrow (M^{**}\otimes_R N)^{**}
  \]
  is an isomorphism.
\end{lemma}

\begin{proof}
  Replacing $M$ and $N$ by $M\tf$ and $N\tf$, respectively, we may assume
  that $M$ and $N$ are lattices.
  By Lemma~\ref{tensor-lattice.thm}, 
Lemma~\ref{intersection.thm}, 
and Lemma~\ref{Hom-reflexive.thm},
it suffices to show that
  for any height one prime $P$ of $R$, 
  \[
  ((M\otimes_R N)\tf)_P\rightarrow ((M^{**}\otimes_R N)\tf)_P
  \]
  is an isomorphism.
  This is equivalent to say that
  \[
  M_P\otimes_{R_P}N_P\rightarrow (M^{**})_P\otimes_{R_P}N_P
  \]
  is an isomorphism.
  This is trivial.
\end{proof}

\paragraph\label{lattice.par}
Let $X$ be a scheme.
We say that $X$ is locally integral (resp.\ locally Krull) 
if there exists some affine open covering
$X=\bigcup_{i\in I}\Spec A_i$ with each $A_i$ a domain (resp.\ Krull domain).
A locally Krull scheme is locally integral.
A locally integral scheme is a disjoint union
$X=\bigcup_{j\in J}X_j$ with each $X_j$ an integral closed open subscheme.
If $X$ is locally Krull and $U=\Spec A$ is an affine open subset with
$A$ a domain, then $A$ is a Krull domain, as can be seen easily from
Lemma~\ref{Krull-descent.thm}.

\paragraph
Let $X$ be a locally integral scheme.
An $\O_X$-module $\M$ is called a lattice or $\O_X$-lattice 
if $\M$ is quasi-coherent, and
for any affine open subset $U=\Spec A$ of $X$ with $A$ an integral domain, 
$\Gamma(U,\M)$ is an $A$-lattice.
This is equivalent to say that there exists some affine open covering
$X=\bigcup_{i\in I}U_i$ such that 
each $A_i=\Gamma(U_i,X)$ is an integral domain and $\Gamma(U_i,\M)$ is
an $A_i$-lattice.
An $\O_X$-module $\M$ is said to be {\em reflexive} if $\M$ is an
$\O_X$-lattice and the canonical map $\M\rightarrow\M^{**}$ is
an isomorphism.
For a quasi-coherent $\O_X$-module $\M$, set $\M\tf=\M/\M\tor$, where
$\M\tor$ is the torsion part of $\M$.
A lattice $\M$ is said to be of rank $n$ if for any point $\xi$ of $X$
such that $\O_{X,\xi}$ is a field, $\M_\xi$ is an $n$-dimensional 
$\O_{X,\xi}$-vector space.

\begin{lemma}
  Let $X$ be a locally Krull scheme, and $\M$, $\N$, $\F$, and $\G$  be 
  quasi-coherent $\O_X$-modules.
  Assume that $\M\tf$ is a lattice, $\N$ is reflexive, and
  $\F$ and $\G$ are flat.
  Then
  \begin{enumerate}
  \item[\bf 1] For any flat morphism $\varphi:Y\rightarrow X$, 
    the canonical map
    \[
    P: \varphi^*\uHom_{\O_X}(\M,\N\otimes_{\O_X} \F)
    \rightarrow
    \uHom_{\O_Y}(\varphi^*\M,\varphi^*\N\otimes_{\O_Y}\varphi^*\F)
    \]
    is an isomorphism.
  \item[\bf 2] $\uHom_{\O_X}(\M,\N\otimes_{\O_X} \F)$ is quasi-coherent.
  \item[\bf 3] The canonical map
    \[
    \uHom_{\O_X}(\M,\N\otimes_{\O_X} \F)\otimes_{\O_X}\G
    \rightarrow
    \uHom_{\O_X}(\M,\N\otimes_{\O_X} \F\otimes_{\O_X}\G)
    \]
    is an isomorphism.
  \end{enumerate}
\end{lemma}

\begin{proof}
  Obvious by Lemma~\ref{Krull-Hom.thm}.
\end{proof}

\begin{lemma}
  Let $G$ be a flat $S$-group scheme, 
  $X$ be a $G$-scheme, and $\M$ and $\N$ be quasi-coherent
  $(G,\O_X)$-modules.
  If for any flat $S$-morphism $\varphi:Y\rightarrow X$, 
  the canonical map
  \[
  P: \varphi^*\uHom_{\O_X}(\M,\N)
  \rightarrow
  \uHom_{\O_Y}(\varphi^*\M,\varphi^*\N)
  \]
  is an isomorphism, then the $(G,\O_X)$-module $\uHom_{\O_X}(\M,\N)$ is
  quasi-coherent.
\end{lemma}

\begin{proof}
  Clearly, $\uHom_{\O_X}(\M,\N)=\uHom_{\O_{B_G^M(X)}}(\M,\N)_{[0]}$ is
  quasi-coherent.

  By \cite[(6.37)]{ETI},
  \[
  \alpha_\phi:(B_G^M(X))_\phi^*\uHom_{\O_{B_G^M(X)}}(\M,\N)_{[0]}
  \rightarrow
  \uHom_{\O_{B_G^M(X)}}(\M,\N)_j
  \]
  is an isomorphism for any $j\in\ob(\Delta_M)=\{[0],[1],[2]\}$ and
  $\phi:[0]\rightarrow j$.
  So $\uHom_{\O_{B_G^M(X)}}(\M,\N)_j$ is quasi-coherent for every $j$, and
  hence $\uHom_{\O_{B_G^M(X)}}(\M,\N)$ is locally quasi-coherent 
  (this is the precise meaning of saying that $\uHom_{\O_X}(\M,\N)$ is locally
  quasi-coherent).
  On the other hand, $\uHom_{\O_{B_G^M(X)}}(\M,\N)$ is equivariant by 
  \cite[(7.6)]{ETI}.
  By \cite[(7.3)]{ETI}, $\uHom_{\O_X}(\M,\N)$, or better, 
$\uHom_{\O_{B_G^M(X)}}(\M,\N)$ is quasi-coherent.
\end{proof}

\begin{corollary}\label{hom.thm}
  Let $G$ and $X$ be as above, and $\M$, $\N$ and $\P$ be quasi-coherent 
  $(G,\O_X)$-modules.
  Assume that $X$ is locally Krull, $\M\tf$ is a lattice,
  $\N$ reflexive, and $\P$ flat.
  Then the $(G,\O_X)$-module $\uHom_{\O_X}(\M,\N\otimes_{\O_X}\P)$ is
  quasi-coherent.
\qed
\end{corollary}

\paragraph\label{equivariant-class.par}
Let $Y$ be a locally Krull scheme.
We denote the set of isomorphism classes of rank-one reflexive sheaves 
by $\Cl(Y)$, and call it the class group of $Y$.
Let $G$ be a flat $S$-group scheme, 
$X$ be a $G$-scheme which is locally Krull.
A quasi-coherent $(G,\O_X)$-module which is reflexive (of rank $n$) as
an $\O_X$-module is simply called a reflexive $(G,\O_X)$-module (of
rank $n$).
We denote the set of isomorphism classes of rank-one reflexive 
$(G,\O_X)$-modules by 
$\Cl(G,X)$, and call it the $G$-equivariant class group 
of $X$.
There is an obvious map $\alpha:\Cl(G,X)\rightarrow \Cl(X)$, forgetting
the $G$-action.
By Lemma~\ref{tensor-lattice.thm}, Lemma~\ref{associativity.thm} and
Corollary~\ref{hom.thm}, defining
\[
[\M]+[\N]=[(\M\otimes_{\O_X}\N)^{**}],
\]
$\Cl(G,X)$ and $\Cl(Y)$ are abelian (additive) groups, and 
$\alpha$ is a homomorphism.

Note that $\Pic(G,X)$ is a subgroup of $\Cl(G,X)$, and 
$\Pic(Y)$ is a subgroup of $\Cl(Y)$.
Note that $\Ker\alpha=\Ker\rho$, where $\rho:\Pic(G,X)\rightarrow\Pic(X)$
is the map forgetting the $G$-action, as before.

\begin{lemma}\label{reflexive-ascent.thm}
Let $\varphi:X\rightarrow Y$ be a flat morphism of schemes.
Assume that $X$ and $Y$ are locally integral.
If $\M$ is an $\O_Y$-lattice, then $\varphi^*\M$ is an $\O_X$-lattice.
If $Y$ is locally Krull and $\M$ is a reflexive $\O_Y$-module, then
$\varphi^*\M$ is reflexive.
\end{lemma}

\begin{proof}
Follows from Lemma~\ref{basic2.thm} and Lemma~\ref{flat-reflexive.thm}.
\end{proof}

\begin{lemma}\label{reflexive-descent2.thm}
Let $\varphi:X\rightarrow Y$ be an fpqc morphism
of schemes, and assume that $X$ is locally Krull.
Then $Y$ is locally Krull.
If $\M$ is a quasi-coherent $\O_Y$-module such that $\varphi^*\M$ is an
$\O_X$-lattice \(resp.\ reflexive $\O_X$-module\), then $\M$ is an
$\O_Y$-lattice \(resp.\ reflexive $\O_Y$-module\).
\end{lemma}

\begin{proof}
The first assertion is an immediate consequence of 
Lemma~\ref{Krull-descent.thm}.

The second assertion follows from
Lemma~\ref{reflexive-descent.thm}.
\end{proof}

\paragraph 
Let $G$ be a flat $S$-group scheme, and $X$ be a locally Krull $G$-scheme.
We denote the category of reflexive $(G,\O_X)$-modules by $\Ref(G,X)$.
Its full subcategory consisting of reflexive $(G,\O_X)$-modules of rank $n$
is denoted by $\Ref_n(G,X)$.

If we do not consider a $G$-action, $\Ref(X)$ and $\Ref_n(X)$ are defined
similarly.

\begin{lemma}\label{pfb-fpqc.thm}
Let $G$ be an $S$-group scheme, $\varphi:X\rightarrow Y$ be a 
principal $G$-bundle such that the second projection $G\times X\rightarrow X$ 
is flat.
Then $\varphi$ is fpqc.
\end{lemma}

\begin{proof}
There is an fpqc map $h:Y'\rightarrow Y$ such that the base change 
$X'\rightarrow Y'$ is 
a trivial $G$-bundle.
Then $h$ is the composite of
\[
Y'\xrightarrow{e}G\times Y'\cong X'=Y'\times_Y X\xrightarrow{p_2}
 X\xrightarrow {\varphi} Y,
\]
and it factors through $X$.
As $G\times X\rightarrow X$ is flat, $G\times Y'\rightarrow Y'$ is also
flat.
Thus $X'\rightarrow Y'$ is flat, and hence so is $\varphi:X\rightarrow Y$
by descent.

Next, take a quasi-compact open subset $U$ of $Y$.
There exists some quasi-compact open subset $V$ of $Y'$ such that
$h(V)=U$.
As the image $W$ of $V$ in $X$ is quasi-compact, there exists some
quasi-compact open subset $W'$ of $\varphi^{-1}(U)$ such that $W\subset W'$.
Then $U=\varphi(W)\subset \varphi(W')\subset \varphi(\varphi^{-1}(U))\subset U$,
and hence $\varphi(W')=U$.

This shows that $\varphi$ is fpqc.
\end{proof}

\begin{lemma}\label{codim-two-isom.thm}
Let $X$ be a locally Krull scheme.
Let $U$ be its open subset.
Let $\varphi:U\hookrightarrow X$ be the inclusion.
Assume that $\codim_X(X-U)\geq 2$.
Then $\varphi^*:\Ref_n X\rightarrow\Ref_nU$ is an equivalence,
and $\varphi_*:\Ref_nU\rightarrow\Ref_nX$ is its quasi-inverse.
\end{lemma}

\begin{proof}
By Lemma~\ref{reflexive-ascent.thm}, 
$\varphi^*:\Ref_n(G,X)\rightarrow\Ref_n(G,U)$ is well-defined.
Thus it suffices to show that 
$\varphi_*:\Ref(G,U)\rightarrow \Ref(G,X)$ is well-defined, and is a 
quasi-inverse to $\varphi^*$.
That is, for $\N\in\Ref(G,U)$, $\varphi_*\N\in\Ref(G,X)$, and
for $\M\in\Ref(G,X)$, the canonical map $\M\rightarrow \varphi_*\varphi^*\M$ 
is an isomorphism.

The question is local on $X$, and we may assume that $X=\Spec A$ is
affine and integral.
Then $U=X\setminus V(I)$ for some ideal $I$ of $A$ such that $\height I
\geq 2$, where $V(I)=\{P\in\Spec A\mid P\supset I\}$.
We can take a finitely generated ideal $J\subset I$ 
such that $\height J
\geq 2$.
Set $W=X\setminus V(J)$.
It suffices to show the assertion in problem for $W\rightarrow U$ and
$W\rightarrow X$.
So replacing $U$ by $W$ (and changing $X$), we may assume that the 
open immersion $U\rightarrow X$ is quasi-compact.
Replacing $X$ again if necessary, we may assume that $X=\Spec A$ is 
affine and integral.

Now $\varphi$ is concentrated, and hence $\varphi_*\N$ is quasi-coherent.
Let $\eta$ be the generic point of $X$.
Let $U=\bigcup_{i=1}^r U_i$, where $U_i=\Spec A[1/f_i]$ with $f_i\in A\setminus
0$.
Then $\Gamma(U_i,\N)\subset M_i\subset \N_\eta$ for some finitely generated
$A[1/f_i]$-module $M_i$.
Let $m_{i1},\ldots,m_{is_i}\in M_i$ be the generators of $M_i$.
Let $M$ be the $A$-span of $\{m_{ij}\mid 1\leq i\leq r,1\leq j\leq s_i\}$,
and $\M$ the associated sheaf of the $A$-module $M^{**}$ on $X=\Spec A$.
As $\N_P\subset \M_P$ for height one prime ideal of $A$, 
\[
\Gamma(X,\varphi_*\N)=\Gamma(U,\N)=\bigcap_{\height P=1,\;P\in X}\N_P
\subset \bigcap_{\height P=1,\;P\in X}\M_P=\Gamma(X,\M),
\]
and $\varphi_*\N\subset \M$.
Thus $\varphi_*\N$ is a lattice.

Set $N=\Gamma(X,\varphi_*\N)$.
It remains to show that $N$ is a reflexive $A$-module.
This is easy, since
\[
N=\Gamma(U,\N)=\bigcap_{\height P=1,\;P\in U}\N_P
=\bigcap_{\height P=1}(\varphi_*\N)_P =\bigcap_{\height P =1} N_P
\]
by the reflexive property of $\Cal N$ and the
quasi-coherence of $\varphi_*\Cal N$.

Finally, we prove that for $\M\in\Ref_n(X)$, 
$\M\rightarrow \varphi_*\varphi^*\M$ is an isomorphism.
As this is an $O_X$-linear map between quasi-coherent $\O_X$-modules,
it suffices to show that $\Gamma(X,\M)\rightarrow \Gamma
(X,\varphi_*\varphi^*\M)$ is
an isomorphism.
By Lemma~\ref{intersection.thm}, 
\[
\Gamma(X,\M)=\bigcap_{\height P=1,\,P\in X}\M_P\quad\text{and}\quad
\Gamma(X,\varphi_*\varphi^*\M)=\bigcap_{\height P=1,\,P\in U}\M_P,
\]
so they are equal, and we are done.
\end{proof}

\begin{lemma}\label{codim-two.thm}
Let $Y$ be a quasi-compact locally Krull scheme, and $U$ its open subset.
Then there exists some quasi-compact open subset $V$ of $U$ such that
$\codim_U(U\setminus V)\geq 2$.
\end{lemma}

\begin{proof}
Let $Y=\bigcup_i Y_i$ with $Y_i$ a spec of a Krull domain.
Then replacing $Y$ with $Y_i$ and $U$ with $Y_i\cap U$, we may assume
that $Y=\Spec A$ with $A$ a Krull domain.
Then there is a radical ideal $I$ of $A$ such that
$U=D(I):=Y\setminus V(I)$.
Take $a\in I\setminus 0$.
Let $\Min(Aa)\setminus V(I)=\{P_1,\ldots,P_r\}$.
Take $b_i\in I\setminus P_i$, and set $J=(a,b_1,\ldots,b_r)$.
Then $\Min(J)\cap X^1(A)\subset V(I)$.
So letting $V=D(J)$, $\codim_U(U\setminus V)\geq 2$.
As $J$ is finitely generated, $V$ is quasi-compact.
\end{proof}

\begin{lemma}\label{locally-Krull-qc.thm}
Let $G$ be a flat $S$-group scheme.
Let $\varphi:U\rightarrow Y$ be a quasi-separated $G$-morphism.
Assume that there exists a factorization $\varphi=\psi h$ such that
$h:U\rightarrow X$ is an open immersion, $\psi:X\rightarrow Y$ 
is quasi-compact, and $X$ is locally Krull \(we do not requre that
$G$ acts on $X$\).
Then for any reflexive $(G,\O_U)$-module $\M$, $\varphi_*\M$
is a quasi-coherent $(G,\O_Y)$-module.
\end{lemma}

\begin{proof}
Let $\M_{[0]}$ be the associated $\O_U$-module of $\M$.
We show that $\varphi_*\M_{[0]}$ is quasi-coherent.
In order to do so, we may assume that $G$ is trivial.
Then the question is local on $Y$, we may assume that $Y$ is
affine.
Now by Lemma~\ref{codim-two.thm}, 
we can take a quasi-compact open subscheme $V$ of $U$ such that
$\codim_U(U\setminus V)\geq 2$.
Let $i:V\rightarrow U$ be the inclusion.
Then $\M\cong i_*i^*\M$ by Lemma~\ref{codim-two-isom.thm}.
So we may assume that $U$ itself is quasi-compact.
Then $\varphi$ is quasi-compact quasi-separated, and hence
$\varphi_*\M$ is quasi-coherent by \cite[(9.2.1)]{EGA-I}, as required.

Next we show that for any flat $Y$-scheme $f:F\rightarrow Y$, 
Lipman's theta $\theta:f^*\varphi_*\M\rightarrow (p_2)_*p_1^*\M$ 
(see for the definition, \cite[(3.7.2)]{Lipman} and \cite[(1.21)]{ETI})
is
an isomorphism, where $p_1:U\times_Y F\rightarrow U$ and
$p_2:U\times_Y F\rightarrow F$ are projection maps.
Again, $G$ is irrelevant here, and we may assume that $Y$ is affine.
Take $V$ as above, and consider the commutative diagram
\[
\xymatrix{
V\times_Y F \ar[r]^{i\times 1} \ar[d]^{q_1}
\ar@{}[dr]|{\text{\normalsize $\tau$}}
& 
U\times_Y F \ar[d]^{p_1} \ar[r]^{p_2} 
\ar@{}[dr]|{\text{\normalsize $\sigma$}}
& 
F \ar[d]^f \\
V \ar[r]^i
& U \ar[r]^\varphi & Y
}.
\]
By \cite[(3.7.2)]{Lipman}, it suffices to prove that
\[
\theta(\tau): p_1^*i_*\N\rightarrow (i\times 1)_*q_1^*\N
\]
and
\[
\theta(\tau+\sigma): f^*(\varphi i)_*\N\rightarrow (p_2(i\times 1))_*q_1^*\N
\]
are isomorphisms, where $\N=i^*\M$.
Replacing $U$ by $V$ and $\M$ by $\N$,
it is easy to see that we may assume that $Y$ is affine and $U$ is
quasi-compact.
This case is \cite[(3.9.5)]{Lipman} (see also \cite[(7.12)]{ETI}).

Now consider the original problem.
As we have seen, Lipman's theta
$\theta: (B_G^M(Y)_\phi)^*\varphi_*\M_{[0]}\rightarrow
(B_G^M(\varphi)_{[j]})_*B_G^M(U)_\phi^*\M_{[0]}$ is an isomorphism for
any morphism $\phi:[0]\rightarrow[j]$ in $\Delta_M$.
In particular, letting $j=1,2$ and taking any $\phi:[0]\rightarrow [j]$, 
we have that $\varphi_*\M$ (which is officially $B_G^M(\varphi)_*\M$) is
locally quasi-coherent.
Indeed, we already know that $\varphi_*\M_{[0]}$ is quasi-coherent,
and $B_G^M(U)_\phi^*\M_{[0]}\cong \M_{[j]}$ by the equivariance of $\M$.

Moreover, by \cite[(6.20)]{ETI}, the alpha map 
$\alpha_\phi: B_G^M(Y)_\phi^* (B_G^M(\varphi)_*\M)_{[0]}\rightarrow 
(B_G^M(\varphi)_*\M)_{[j]}$ is an isomorphism for any $[j]\in\{[0],[1],[2]\}$ 
and any $\phi:[0]\rightarrow [j]$.
By \cite[(7.6), {\bf 3}]{ETI}, $\varphi_*\M$ is equivariant.
Hence $\varphi_*\M$ is quasi-coherent by \cite[(7.3)]{ETI}, as desired.
\end{proof}

\begin{corollary}\label{codim-two-ref.thm}
Let $G$ be a flat $S$-group scheme, and
$X$ be a locally Krull $G$-scheme.
Let $U$ be its $G$-stable open subset.
Let $\varphi:U\hookrightarrow X$ be the inclusion.
Assume that $\codim_X(X\setminus U)\geq 2$.
Then $\varphi^*:\Ref_n (G,X)\rightarrow\Ref_n(G,U)$ is an equivalence,
and $\varphi_*:\Ref_n(G,U)\rightarrow\Ref_n(G,X)$ is its quasi-inverse.
In particular, $\varphi^*:\Cl(G,X)\rightarrow\Cl(G,U)$ defined by
$\varphi^*[\M]=[\varphi^*\M]$ is an isomorphism whose inverse is given
by $\N\mapsto [\varphi_*\N]$.
\qed
\end{corollary}

\begin{proof}
By Lemma~\ref{locally-Krull-qc.thm}, $\varphi_*$ is a functor from
$\Ref(G,U)$ to $\Qch(G,X)$.
The rest is easy by Lemma~\ref{codim-two-isom.thm}.
\end{proof}

\begin{proposition}\label{pfb-cl-isom.thm}
Let $G$ be a flat $S$-group scheme, and
$\varphi:X\rightarrow Y$ a principal $G$-bundle.
Then $\varphi$ is fpqc.
If $X$ is locally Krull, then $Y$ is also locally Krull.
The equivalence $\varphi^*:\Qch(Y)\rightarrow\Qch(G,X)$ yields
an equivalence $\varphi^*:\Ref_n(Y)\rightarrow\Ref_n(G,X)$.
In particular, $\varphi^*:\Cl(Y)\rightarrow\Cl(G,X)$ is an isomorphism.
\end{proposition}

\begin{proof}
The first assertion is by Lemma~\ref{pfb-fpqc.thm}.
Assume that $X$ is locally Krull.
Then $Y$ is locally Krull by Lemma~\ref{reflexive-descent2.thm}.
The equivalence $\varphi^*:\Qch(Y)\rightarrow\Qch(G,X)$ is
by Lemma~\ref{pfb-pic.thm}.
For $\M\in\Qch(Y)$, $\M\in\Ref_n(Y)$ if and only if $\varphi^*\M\in\Ref_n(G,X)$
by Lemma~\ref{reflexive-ascent.thm} and Lemma~\ref{reflexive-descent2.thm}.
The last assertion is now trivial.
\end{proof}

\begin{proposition}\label{Cl-Pic.thm}
Let $Y$ be a quasi-compact locally Krull scheme.
Then $\Cl(Y)\cong \indlim \Pic(U)$, where the inductive limit is
taken over all open subsets $U$ such that $\codim_Y(Y\setminus U)\geq 2$.
\end{proposition}

\begin{proof}
By Corollary~\ref{codim-two-ref.thm} for the case that $G$ is trivial,
the map $\Cl(Y)\rightarrow \indlim \Cl(U)$ is an isomorphism.
So it suffices to show that the canonical map $\indlim \Pic(U)\rightarrow
\indlim \Cl(U)$ is surjective, as the injectivity is obvious.
This amounts to show that, for each $U$ and a rank-one reflexive sheaf 
$\M$ over $U$, there exists some open subset $V$ of $U$ such that 
$\codim_U (U\setminus V)\geq 2$ and $\M|_V$ is an invertible sheaf.

By Lemma~\ref{codim-two.thm}, there exists some quasi-compact 
open subset $U'$ of $U$ such that $\codim_U(U\setminus U')\geq 2$.
Replacing $U$ by $U'$, we may assume that $U$ is quasi-compact.
Then $U=\bigcup_i \Spec A_i$ with $A_i$ a Krull domain.
Replacing $U$ by each $\Spec A_i$, we may
assume that $U=\Spec A$ is affine with $A$ a Krull domain.
Set $I:=\Gamma(U,\M)$.
We may assume that $I$ is a divisorial ideal of $A$.
Take $a\in I\setminus \{0\}$.
Let $\{P_1,\ldots,P_r\}$ be the set of minimal primes of $Aa$.
We may assume that $P_i\neq P_j$ for $i\neq j$.
Let $1\leq i\leq r$.
Set $IA_{P_i}=P_i^{v_i}A_{P_i}$.
For each $i$, take $b_i\in I\setminus P_i^{v_i+1}A_{P_i}$.
Set $J=Aa+\sum_{i=1}^r(Ab_i:I)$.
If $P\neq P_i$ for any $i$, $J_P=A_P$, since $(Aa)_P=A_P$.
Moreover, $J_{P_i}=A_{P_i}$, since $(Ab_i:I)_{P_i}=(Ab_i)_{P_i}:I_{P_i}=A_{P_i}$.
Set $V=D(J)=U\setminus V(J)$.
Then $\codim_U(U\setminus V)\geq 2$.
On $D(Aa)$, $\tilde I|_{D(Aa)}=\tilde A|_{D(Aa)}$ is an invertible sheaf,
where $D(Aa)=\Spec A\setminus V(Aa)$.
On $D(Ab_i:I)$, $\tilde I|_{D(Ab_i:I)}=(Ab_i)\,\tilde{}\,|_{D(Ab_i:I)}$ 
is an invertible sheaf.
Thus $\tilde I$ is invertible on $V$, and we are done.
\end{proof}

\begin{lemma}\label{invariance-reflexive.thm}
Let $G$ be a flat $S$-group scheme, 
and $X$ be a locally Krull $S$-scheme on which $G$ acts trivially.
Let $\M\in\Ref(G,X)$.
Then $\M^G\in\Ref(G,X)$.
\end{lemma}

\begin{proof}
Let $p:G\times X\rightarrow X$ be the second projection.
There is an exact sequence
\[
0\rightarrow \M^G \xrightarrow{i} \M \rightarrow p_*p^*\M.
\]
By Lemma~\ref{second-syzygy.thm}, it suffices to show that
the cokernel $\C$ of $i$ is torsion-free.
As $p$ is flat and $\C$ is a subsheaf of $p_*p^*\M$, this is easy.
\end{proof}

\section{The class group of an invariant subring}

\begin{lemma}\label{finite-direct-Krull.thm}
Let $X$ be a quasi-compact locally Krull scheme, and $U$ its open subscheme.
Then $\Gamma(U,\O_U)$ is a finite direct product of Krull domains.
\end{lemma}

\begin{proof}
As $U$ is a finite direct product of integral schemes, we may assume
that $U$ is integral.
By Lemma~\ref{codim-two.thm}, we can take a quasi-compact 
open subset $V$ of $U$ such that $\codim_U(U\setminus V)\geq 2$.
Replacing $U$ by $V$, we may assume that $U$ itself is quasi-compact.
If $U=\bigcup_{i=1}^n U_i$ with $U_i$ affine, then
$\Gamma(U,\O_U)=\bigcap_{i=1}^n \Gamma(U_i,\O_{U_i})$ with each 
$\Gamma(U_i\O_{U_i})$ a Krull domain, and hence $U$ is also a 
Krull domain.
\end{proof}

\paragraph 
Let $G$ be a flat $S$-group scheme.
Let $X$ be a quasi-compact quasi-separated locally Krull $G$-scheme,
and let $\varphi:X\rightarrow Y$ be a $G$-invariant morphism such that
$\O_Y\rightarrow(\varphi_*\O_X)^G$ is an isomorphism.

\begin{lemma}\label{Y-Krull.thm}
$Y$ is a locally Krull scheme.
Each irreducible component of $X$ is mapped dominatingly to an 
irreducible component of $Y$.
In particular, $Y$ has only finitely many irreducible components.
Moreover, there exists some quasi-compact open subset $U$ of $Y$ 
such that $\codim_Y (Y\setminus U)\geq 2$.
\end{lemma}

\begin{proof}
Let $Y'=\Spec A$ be an affine open subscheme of $Y$,
$X'=\varphi^{-1}(Y')$, and $\varphi':X'\rightarrow Y'$
be the induced map.

Let $B=\Gamma(X',\O_{X'})$.
Note that $B$ is a finite direct product of Krull domains by
Lemma~\ref{finite-direct-Krull.thm}.
Note also that the sequence
\begin{equation}\label{ABG.eq}
0\rightarrow A \rightarrow B \xrightarrow{u-v}C
\end{equation}
is exact, where $C=\Gamma(G\times X',\O_{G\times X'})$, 
and $u=u(a)$ and $v=u(p_2)$ are the maps $B=\Gamma(X',\O_{X'})
\rightarrow \Gamma(G\times X',\O_{G\times X'})=C$ corresponding to the action 
$a$ and the second projection $p_2$, respectively.
As in the proof of \cite[(32.6)]{ETI}, a nonzerodivisor of $A$ 
is a nonzerodivisor of $B$, $A=Q(A)\cap B$, and 
hence $A$ is a finite direct product of Krull domains.
Also, as any nonzerodivisor of $A$ is a nonzerodivisor of $B$,
any irreducible component of $X$ is mapped dominatingly to $Y$.

We prove the last assertion.
Let $Y=\bigcup_\lambda U_\lambda$ be an affine open covering.
Then by the quasi-compactness of $X$, there are finitely many 
$\lambda_1,\ldots,\lambda_n$ such that $X=\bigcup_i \varphi^{-1}(U_{\lambda_i})$.
Set $U=\bigcup_i U_{\lambda_i}$.
We prove that $\codim_Y(Y\setminus U)\geq 2$.
Assume the contrary, and take $y\in Y\setminus U$ such that $\O_{Y,y}$ is 
a DVR.
Take an affine open neighborhood $Y'=\Spec A$
and let $X':=Y'\times_Y X$.
Then we have the exact sequence (\ref{ABG.eq}) with
$B=\Gamma(X',\O_{X'})$ and $C=\Gamma(G\times X',\O_{X'})$.
Set $Y''=\Spec A_P=\Spec \O_{Y,y}$, where $P$ is the height-one prime ideal
of $A$ corresponding to $y$.
Then plainly,
\[
0\rightarrow A_P\rightarrow B_P\xrightarrow{u-v} C_P
\]
is exact.
Let $t$ be the prime element of $A_P$.
As $\varphi^{-1}(y)$ is empty, $t\O_{X''}=\O_{X''}$, where $X''=Y''\times_Y X$.
Thus $t\in\Gamma(X'',\O_{X''})^\times$.
As there is a quasi-compact open subset $W$ of $X'$ with $\codim_{X'}(X'
\setminus W)\geq 2$, 
\[
t^{-1}\in\Gamma(X'',\O_{X''})=\Gamma(Y''\times_{Y'}W,\O_{Y''\times_{Y'}W})
=\Gamma(W,\O_W)_P=B_P.
\]
So $t^{-1}\in B_P\cap Q(A)=A_P$, and this is a contradiction.
\end{proof}

\begin{lemma}\label{subquotient.thm}
The class group $\Cl(Y)$ of $Y$ is a subquotient of $\Cl(G,X)$.
\end{lemma}

\begin{proof}
By Lemma~\ref{Y-Krull.thm}, there exists some quasi-compact open subset 
$Y'$ of $Y$ such that $\codim_Y(Y\setminus Y')\geq 2$.

Let $h:\Cl(G,X)\rightarrow\indlim \Cl(G,\varphi^{-1}(U))$ be the canonical
map, where the inductive limit is taken over all open subset $U$ of $Y'$
such that $\codim_Y(Y\setminus U)\geq 2$.
Let $\nu:\Cl(Y)\rightarrow \Image h$ be the map defined by 
$\nu[\M]=h[(\varphi^*\M)^{**}]$.
As $\M|_U$ is an invertible sheaf for some $U$, it is easy to see that
$\nu$ is a group homomorphism.
If $\nu[\M]=0$, then $\M|_U$ is an invertible sheaf and $\varphi^*(\M|_U)$ is
trivial for some $U$.
By Lemma~\ref{pic-injective.thm}, $\M|_U$ is trivial, and by
Proposition~\ref{Cl-Pic.thm}, $[\M|_{Y'}]=0$ in $\Cl(Y')$.
By Corollary~\ref{codim-two-ref.thm}, $[\M]=0$ in $\Cl(Y)$.
This shows that $\nu$ is injective, and $\Cl(Y)$ is a subquotient of
$\Cl(G,X)$.
\end{proof}

\begin{theorem}\label{main2.thm}
Let $k$ be a field, $G$ a smooth $k$-group scheme of finite type,
and $X$ a quasi-compact quasi-separated locally Krull $G$-scheme.
Assume that there is a $k$-scheme $Z$ of finite type and a
dominating $k$-morphism $Z\rightarrow X$.
Let $\varphi:X\rightarrow Y$ be a $G$-invariant morphism such that
$\O_Y\rightarrow (\varphi_*\O_X)^G$ is an isomorphism.
Then $Y$ is locally Krull.
If, moreover, $\Cl(X)$ is finitely generated, then $\Cl(G,X)$ and $\Cl(Y)$ are
also finitely generated.
\end{theorem}

\begin{proof}
$Y$ is locally Krull by Lemma~\ref{Y-Krull.thm}.
We prove the last assertion.
If $\Cl(X)$ is finitely generated, then $\Cl(G,X)$ is also finitely generated,
since the kernel of the canonical map $\alpha:\Cl(G,X)\rightarrow
\Cl(X)$ agrees with $\Ker\rho$, which is finitely generated by
Theorem~\ref{main.thm}.
As $\Cl(Y)$ is a subquotient of $\Cl(G,X)$, it is also finitely generated.
\end{proof}

\begin{remark}
A similar result can be found in \cite{Waterhouse}.
\end{remark}

Finally, as a normal scheme of finite type over $k$ is quasi-compact
quasi-separated locally Krull (and is dominated by some scheme of
finite type), we have

\begin{corollary}\label{main2-cor.thm}
Let $k$ be a field, $G$ a smooth $k$-group scheme of finite type,
acting on a normal $k$-scheme $X$ of finite type.
Let $\varphi:X\rightarrow Y$ be a $G$-invariant morphism such that
$\O_Y\rightarrow (\varphi_*\O_X)^G$ is an isomorphism.
Then $Y$ is locally Krull.
If, moreover, $\Cl(X)$ is finitely generated, then $\Cl(G,X)$ and $\Cl(Y)$ are
also finitely generated.
\qed
\end{corollary}


\begin{thebibliography}{GHKN}
\def\ji#1#2(#3)#4-#5.{\newblock{\em#1} {\bf#2} (#3), #4--#5.}
\def\GTM#1{Graduate Texts in Math. {\bf #1}, Springer}
\def\SLN#1{Lecture Notes in Math. {\bf #1}, Springer}

\bibitem[Bor]{Borel}
A. Borel,
{\em Linear Algebraic Groups,} 2nd ed.
Graduate Texts in Math. {\bf 126}, Springer (1991).

\bibitem[dJ]{SP}
J. de Jong et al,
Stacks Project,
pdf version, 
{\tt
http://stacks.math.columbia.edu
}

\bibitem[Dol]{Dolgachev}
I. Dolgachev,
{\em Lectures on Invariant Theory,}
London Math. Soc. Lecture Note Series {\bf 296},
Cambridge (2003).

\bibitem[EKW]{EKW}
E. J. Elizondo, K. Kurano, and K.-i. Watanabe,
The total coordinate ring of a normal projective variety,
{\em J. Algebra} {\bf 276} (2004), 625--637.

\bibitem[Fos]{Fossum}
R. M. Fossum,
{\em The Divisor Class Group of a Krull Domain,}
Springer (1973).

\bibitem[Gro]{EGA-I}
A. Grothendieck, {\em El\'ements de G\'eom\'etrie Alg\'ebrique
I,} IHES Publ. Math. {\bf 4} (1960).

\bibitem[Gro2]{EGA-IV-1}
A. Grothendieck, {\em El\'ements de G\'eom\'etrie Alg\'ebrique
IV, 1e Partie,} IHES Publ. Math. {\bf 20} (1964).

\bibitem[Gro3]{EGA-IV-2}
A. Grothendieck, {\em El\'ements de G\'eom\'etrie Alg\'ebrique
IV, 2e Partie,} IHES Publ. Math. {\bf 24} (1965).

\bibitem[Has]{ETI}
M. Hashimoto,
Equivariant twisted inverses,
{\em Foundations of Grothendieck Duality for Diagrams
of Schemes} (J. Lipman, M. Hashimoto),
\SLN{1960} (2009), pp.~261--478.

\bibitem[Has2]{Hashimoto}
M. Hashimoto,
Equivariant total ring of fractions and factoriality of rings 
generated by semiinvariants,
{\tt arXiv:1009.5152v2}

\bibitem[HO]{HO}
M. Hashimoto and M. Ohtani,
Equivariant Matlis and the local duality,
{\em J. Algebra} {\bf 324} (2010), 1447--1470.

\bibitem[Lip]{Lipman}
J. Lipman,
Notes on derived functors and Grothendieck duality,
{\em Foundations of Grothendieck Duality for Diagrams
of Schemes} (J.~Lipman, M.~Hashimoto),
\SLN{1960} (2009), pp.~1--259.

\bibitem[Mag]{Magid}
A. R. Magid,
Finite generation of class groups of rings of invariants,
{\em Proc. Amer. Math. Soc.} {\bf 60} (1976), 45--48.

\bibitem[Mat]{CRT} 
H. Matsumura, {\em Commutative Ring Theory,} First paperback
edition, Cambridge (1989).

\bibitem[Muk]{Mukai}
S. Mukai,
Geometric realization of $T$-shaped root systems and counterexamples
to Hilbert's fourteenth problem,
{\em Algebraic Transformation Groups and Algebraic Varieties,}
Springer (2004), pp.~123--129.

\bibitem[Nag]{Nagata} M. Nagata,
On the 14th problem of Hilbert,
{\em Amer. J. Math.} {\bf 81} (1959), 766--772.

\bibitem[Nag2]{Nagata2}
M. Nagata,
{\em Local Rings,} Corrected reprint,
Krieger (1975).

\bibitem[Rei]{Reid}
M. Reid,
Canonical 3-folds,
{\em Joun\'ee de G\'eometrie Alg\'ebrique d'Angers, Juillet 1979,}
Sijthoff \& Noordhoff (1980), pp.~273--310.

\bibitem[Ros]{Rosenlicht}
M. Rosenlicht,
Toroidal algebraic groups,
{\em Proc. Amer. Math. Soc.} {\bf 12} (1961), 984--988.

\bibitem[SH]{SH}
I. Swanson and C. Huneke,
{\em Integral Closure of Ideals, Rings, and Modules,}
London Math. Soc. Lecture Note Series {\bf 336},
Cambridge (2006).

\bibitem[Swe]{Sweedler}
M. Sweedler,
A units theorem applied to Hopf algebras and Amitsur cohomology,
{\em Amer. J. Math.} {\bf 92} (1970), 259--271.

\bibitem[Vis]{Vistoli}
A. Vistoli,
Grothendieck topologies, fibered categories and descent theory,
{\em Fundamental Algebraic Geometry,} B. Fantechi et al. (eds.),
AMS (2005), pp.~1--104.

\bibitem[Wat]{Waterhouse}
W. C. Waterhouse,
Class groups of rings of invariants,
{\em Proc. Amer. Math. Soc.} {\bf 67} (1977), 23--26.


\end{thebibliography}
\end{document}